\documentclass[12pt]{amsart}
\usepackage{amsmath, amsthm, amsfonts, amssymb, mathrsfs, graphicx, stmaryrd}
\usepackage[usenames, dvipsnames]{color}
\usepackage[margin=1in]{geometry}
\usepackage[bookmarks, bookmarksdepth=2, colorlinks=true, linkcolor=blue, citecolor=blue, urlcolor=blue]{hyperref}
\usepackage{enumitem}
\usepackage{tikz-cd}
\usepackage{tikz}
\usepackage{fdsymbol}
\usepackage{mathtools}
\usepackage{caption}

\title[GL-algebras in positive characteristic III]{GL-algebras in positive characteristic III: the divided power algebra}
\author{Karthik Ganapathy}
\address{Department of Mathematics, University of California, San Diego, CA}

\email{\href{mailto:kganapathy@ucsd.edu}{kganapathy@ucsd.edu}}
\urladdr{\url{https://sites.google.com/view/karthik-ganapathy/}}
\usetikzlibrary{%
  matrix,%
  calc,%
  arrows%
}

\newcommand{\cC}{\mathcal{C}}

\newcommand{\cD}{\mathcal{D}}

\newcommand{\rD}{\mathrm{D}}

\newcommand{\bG}{\mathbf{G}}
\newcommand{\cG}{\mathcal{G}}

\newcommand{\bK}{\mathbf{K}}

\newcommand{\rL}{\mathrm{L}}

\newcommand{\bN}{\mathbf{N}}

\newcommand{\rR}{\mathrm{R}}

\newcommand{\fS}{\mathfrak{S}}

\newcommand{\cT}{\mathcal{T}}

\newcommand{\bV}{\mathbf{V}}

\newcommand{\fa}{\mathfrak{a}}

\newcommand{\fb}{\mathfrak{b}}

\newcommand{\fm}{\mathfrak{m}}

\newcommand{\fp}{\mathfrak{p}}

\renewcommand{\phi}{\varphi}

\newcommand{\lw}{{\textstyle \bigwedge}}

\makeatletter
\def\Ddots{\mathinner{\mkern1mu\raise\p@
\vbox{\kern7\p@\hbox{.}}\mkern2mu
\raise4\p@\hbox{.}\mkern2mu\raise7\p@\hbox{.}\mkern1mu}}
\makeatother

\DeclareMathOperator{\wgt}{wt}

\DeclareMathOperator{\im}{im} 
\DeclareMathOperator{\Div}{Div} 
\DeclareMathOperator{\coker}{coker}

\DeclareMathOperator{\cone}{Cone}

\DeclareMathOperator{\ext}{Ext}

\DeclareMathOperator{\Sym}{Sym}

\DeclareMathOperator{\Tor}{Tor}

\DeclareMathOperator{\Spec}{Spec}

\DeclareMathOperator{\Ann}{Ann}

\DeclareMathOperator{\Hom}{Hom}

\DeclareMathOperator{\Mod}{Mod}
\newcommand{\id}{\mathrm{id}}

\newcommand{\pol}{\mathrm{pol}}
\newcommand{\Pol}{\mathbf{Pol}_k}

\newcommand{\tors}{\mathrm{tors}}

\DeclareMathOperator{\FI}{FI}

\DeclareMathOperator{\VI}{VI}
\DeclareMathOperator{\Sh}{\bf{\Sigma}}
\DeclareMathOperator{\nSh}{Sh}
\DeclareMathOperator{\De}{\bf{\Delta}}
\DeclareMathOperator{\Rep}{Rep}

\DeclareMathOperator{\Fec}{Vec}

\DeclareMathOperator{\colim}{colim}
\DeclareMathOperator{\fgen}{fg}
\DeclareMathOperator{\fpre}{fp}

\DeclareMathOperator{\lf}{lf}

\newcommand{\Dr}{D^{(r)}}

\newcommand{\Drs}{D_{[r,s]}}
\newcommand{\Drsp}{D_{[r,s+1]}}
\newcommand{\Jrd}{J_{[r,d]}}
\newcommand{\Jrr}{J_{[r,r]}}
\newcommand{\Jrinfty}{J_{[r,\infty]}}
\newcommand{\Drinfty}{D_{[r,\infty]}}
\newcommand{\mShq}{\Sh_{q}}
\newcommand{\mDeq}{\De_{q}}
\newcommand{\bKq}{{{\bK}_{q}}}
\newcommand{\mKq}{\bKq}
\newcommand{\mKql}{\bK_{q;l}}
\DeclareMathOperator{\maxdeg}{{maxdeg}}

\newcommand{\GL}{\mathbf{GL}}

\newcommand{\mGaq}{\Gamma_{q}}
\newcommand{\mDel}{\De_{q;l}}
\newcommand{\mShu}{\Sh_{u}}

\newcommand{\mGau}{\Gamma_{u}}

\makeatletter
\@addtoreset{equation}{section}
\makeatother

\numberwithin{equation}{section}
\newtheorem{theorem}[equation]{Theorem}

\newtheorem{proposition}[equation]{Proposition}
\newtheorem{lemma}[equation]{Lemma}
\newtheorem{corollary}[equation]{Corollary}

\theoremstyle{definition}
\newtheorem{rmk}[equation]{Remark}
\newenvironment{remark}[1][]{\begin{rmk}[#1] \pushQED{\qed}}{\popQED \end{rmk}}
\newtheorem{eg}[equation]{Example}

\newtheorem{defn}[equation]{Definition}
\newenvironment{definition}[1][]{\begin{defn}[#1]\pushQED{\qed}}{\popQED \end{defn}}

\theoremstyle{plain}
\newtheorem{mainthm}{Theorem}

\makeatletter
\renewcommand{\thesubsection}{%
  \ifnum\c@subsection<1 \@arabic\c@section
  \else \thesection.\@arabic\c@subsection
  \fi
}
\makeatother


\date{}
\begin{document}
\begin{abstract}
In this paper, we study $\GL$-equivariant modules over the infinite-variable divided power algebra $D = \Div(k^{\infty})$. Unlike previously analyzed $\GL$-algebras, the divided power algebra is not noetherian or even finitely generated. We show that $D$ is $\GL$-coherent and prove a ``shift theorem'' for finitely presented $D$-modules. Using this, we obtain a (semi-infinite) semi-orthogonal decomposition of its bounded derived category with one piece corresponding to each Frobenius twist $\Dr$ of $D$. Crucial to our approach is the fact that $D$ is a flat colimit of subalgebras which are $\GL$-noetherian.
\end{abstract}
\maketitle

\section{Introduction}\label{s:intro}
A $\GL$-algebra over an algebraically closed field $k$ is a commutative $k$-algebra with a polynomial action of the infinite-rank general linear group $\GL(\bV)$, where $\bV = \bigcup_{n \in \bN} k^n$. The centrality of $\GL$-algebras and $\GL$-varieties in asymptotic commutative algebra has been firmly established, in large part due to work of Sam--Snowden \cite{ss16gl, ss19gl2} and Bik--Draisma--Eggermont--Snowden \cite{bdes21geo, bdes23uni} (see also \cite{dra19top, ess19gen, bds24charp}).  The best-understood and perhaps the most useful $\GL$-algebra, at least over fields of characteristic zero, is the polynomial algebra $S = \Sym(\bV)$ \cite{cef15fi, ss16gl}. Away from characteristic zero, four inequivalent objects generalize $S$: (1) $\FI$-modules; and $\GL$-equivariant modules over (2) $S = \Sym(\bV)$ itself, (3) the exterior algebra $R = \lw(\bV)$, and (4) the divided power algebra $D = \Div(\bV)$. In characteristic zero, all four categories are equivalent to $\Mod_S$: via Schur--Weyl duality for (1), the transpose functor for (3) \cite{ss12tca}, and a $\GL$-algebra isomorphism $D \cong S$ for (4); these equivalences all break down in characteristic $p$.

$\FI$-modules were treated in detail by Nagpal in his thesis \cite{nag15fi}, and we initiated the study of $\GL$-algebras in positive characteristic with a complete description of $\Mod_R$ \cite{gan22ext} and $\Mod_S$ \cite{gan25pol}. All three of these analyses rest crucially on the respective objects satisfying the noetherian property (see for example, \cite{cefn14fi, ss17grobner, nr19fioi, gan22ext}), a property that readily fails for the divided power algebra since it is not even finitely generated. We nonetheless undertake a systematic study of $\Mod_D$ in this paper and rather unexpectedly, provide a comprehensive picture of $\GL$-equivariant $D$-modules, with many of our results mirroring those for $S$-modules from \cite{gan25pol}. Our approach relies heavily on the $\GL$-representation structure of $D$: the results themselves fail for variants of $D$ that change only the $\GL$-structure of the underlying ring (see Remark~\ref{rmk:DvsBinf}).

Our motivation for studying $\Mod_D$ is twofold. First, we recently gave the first examples of finitely generated $\GL$-algebras that are not $\GL$-noetherian in positive characteristic \cite{gan24non, gan25non}, highlighting the need for a systematic theory beyond the noetherian setting; the divided power algebra $D$ is a natural test case. Second, every positive-degree element of $D$ is nilpotent. Such nilpotent algebras play an outsized role in equivariant commutative algebra, often allowing us to explain higher-slope phenomena in algebraic invariants \cite{gan24lnnr}, yet they are invisible to the geometric methods that have proven so successful \cite{bdes21geo, bds24charp, bdes23uni}. A general algebraic theory is therefore all the more pressing (for another example studied in detail, see \cite{nsy25symm}).

Fix an algebraically closed field $k$ of characteristic $p > 0$. We identify the divided power algebra $D = \Div(\bV)$ with the commutative $k$-algebra generated by elements $x_i^{[t]}$, for $i, t \in \mathbf{N}_{>0}$, subject to the relations $x_i^{[t]}x_i^{[s]} = \binom{t+s}{s} x_i^{[t+s]}$ for all $i, t, s \in \mathbf{N}_{>0}$. This is a graded $k$-algebra with $\deg x_i^{[t]} = t$. 
Throughout the paper, unless specified otherwise, all notions are implicitly $\GL$-equivariant. Thus, a \textit{$D$-module} means a $\GL$-equivariant $D$-module, and such a module is \textit{finitely generated} if it is generated as an ordinary $D$-module by the $\GL$-orbits of finitely many elements, and so on.

\subsection{Main results}
In the theory of $\GL$-algebras, the first object of interest is the $\GL$-spectrum, which plays the role of the Zariski spectrum of a commutative ring (see Section~\ref{ss:glprimes} for definitions). Our first main result is the computation of $\Spec_{\GL}(D)$.

For each $r \in \mathbf{N}_{>0}$, let $I_r$ be the $\GL$-stable ideal of $D$ generated by positive degree elements of degree $\leq p^{r-1}$, and set $I_{\infty}$ to be the homogeneous maximal ideal of $D$ and $I_0 = (0)$. For each $r \in \bN$, the quotient algebra $D/I_r$ is easily seen to be $\Dr$, the $r$-th Frobenius twist of $D$. Let $\overline{\bN} = \bN \cup \{ \infty\}$ endowed with the right-order topology; i.e., the nonempty closed sets are of the form $[n, \infty) \cup \{\infty\}$ with $n \in \overline{\bN}$. 

\begin{mainthm}[$\GL$-spectrum of $D$]\label{thm:introspec}
    The map 
    $$\overline{\bN} \ni n \mapsto I_n \in \Spec_{\GL}(D)$$
    is a
    homeomorphism from $\overline{\bN}$ to the $\GL$-spectrum of $D$.
\end{mainthm}

The algebra $D$ does not satisfy the $\GL$-noetherian property; for instance, the ideal $I_{\infty} \subset D$ is not finitely generated. However, we have:
\begin{mainthm}[Coherence of $D$]\label{thm:coherence}
    The $\GL$-algebra $D$ is $\GL$-coherent, i.e., finitely generated submodules of finitely presented modules are also finitely presented. 
\end{mainthm}
It is worth noting that the $k$-algebra $D$ is also coherent in the usual sense. However, this is orthogonal to Theorem~\ref{thm:coherence}: the usual coherence of $D$ cannot be applied even to the ideals $I_r$ as they are not finitely generated in the usual sense.

The above theorem means that the category of finitely presented $D$-modules is an abelian category. Our paper concerns its bounded derived category $\rD^b_{\fpre}(\Mod_D)$. Much of this paper is dedicated to proving the next result. 
\begin{mainthm}\label{thm:gensdbmodintro}
       The triangulated category $\rD^b_{\fpre}(\Mod_D)$ is generated by the classes of modules $\Dr \otimes L_{\lambda}$ where $r$ ranges over $\bN$ and $\lambda$ over all partitions.
\end{mainthm}
With some more work, we can strengthen this result to obtain a semi-orthogonal decomposition for this triangulated category. The convention for semi-orthogonal decomposition we use is from \cite[Section~4]{ss19gl2}, i.e., given a triangulated category $\cT$, a \textit{semi-orthogonal decomposition} $\cT = \langle \ldots, \cT_1, \cT_0 \rangle$ is a collection of full triangulated subcategories $\cT_i \subset \cT$ such that the smallest triangulated subcategory generated by the collection is $\cT$, and for objects $x \in \cT_i$ and $y \in \cT_j$ with $i > j$, we have $\Hom_{\cT}(x, y) = 0$. 
\begin{mainthm}\label{thm:introsod}
Let $\cT_r$ be the triangulated subcategory of $\rD^b_{\fpre}(\Mod_D)$ generated by modules of the form $\Dr \otimes L_{\lambda}$ with $\lambda$ allowed to vary. We have a semi-orthogonal decomposition
\[
\rD^b_{\fpre}(\Mod_D) = \langle \ldots, \cT_1, \cT_0 \rangle.
\]
\end{mainthm}
Heuristically, each piece $\cT_r$ should be thought of as the derived category of $D$-modules supported exactly in $V(I_r) \subset \Spec_{\GL}(D)$. The essence of Theorem~\ref{thm:introsod} is really that the result holds at the level of the \textit{bounded} derived category. 

\subsection{Overview}\label{ss:strategy}
We now provide a brief outline of how we prove these results.

In Section~\ref{s:prelim}, we recall basic definitions and general results concerning $\GL$-algebras, and apply this to the algebras of interest in Section~\ref{s:dmodules}. In particular, we prove Theorems~\ref{thm:introspec} and~\ref{thm:coherence} by fairly elementary methods. To prove Theorem~\ref{thm:coherence}, first observe that we can write $D$ as a colimit of the finitely generated subalgebras $D_{[0,s]}$ generated by elements of degree $\leq p^s$. Furthermore, the inclusion $D_{[0,s]} \to D$ is flat. The key point is that $D_{[0,s]}$ is $\GL$-noetherian, a result that can be traced back to Cohen \cite{co67laws}. Since $D$ is a flat colimit of $\GL$-noetherian algebras, it is $\GL$-coherent. Crucial to prove Theorem~\ref{thm:introspec}, is the $\GL$-representation structure of the graded pieces of $D$. Specifically, the socle of $\Div_n(\bV)$ is simple and generated by the image of powers of $\Div_1(\bV)$.

Every finitely presented $D$-module has a finite filtration where the successive quotients are torsion-free modules over $D/I_r$ with $r \in \bN$ (Proposition~\ref{prop:devissage}). This suggests the need to simultaneously study $\Dr$-modules for $r \geq 1$. 
We define subalgebras $\Drs \subset \Dr$ generated by elements of degree $\leq p^{s}$ as in the previous paragraph.
Every finitely presented $D$-module (resp.~$\Dr$-module) is obtained by the base change of a finitely presented module over $D_{[0,s]}$ (resp.~$\Drs$) for $s$ sufficiently large. Furthermore, since the inclusion $\Drs \to \Dr$ is flat for all $r, s$, base change along this map preserves much of the relevant homological information. Thus, we are essentially reduced to studying finitely generated $\Drs$-modules.

Section~\ref{s:nagpalshift} forms the technical core of this paper where we systematically develop the structure theory of $\Drs$-modules with $r \leq s < \infty$. The main result is a ``shift theorem" for finitely generated $\Drs$-modules using the Hasse--Schur derivative functors $\{\Sh_m\}_{m \in \bN}$ originally defined in \cite{gan25pol} and reviewed in Sections~\ref{ss:hasseschur} and~\ref{ss:shiftdefn}.
\begin{mainthm}\label{thm:shiftintro}
Given a finitely generated $\Drs$-module $M$, the $\Drs$-module $\Sh_{p^r}^t(M)$ is flat for $t \gg 0$.  
\end{mainthm}
A flat $\Drs$-module is, upon forgetting the $\GL$-action, free as a $\Drs$-module. By an easy induction, the shift theorem allows us to obtain generators of the bounded derived category of $\Mod_{\Drs}$ (Theorem~\ref{thm:gensdbDrs}) in Section~\ref{ss:gensdbDrs}. Proving the shift theorem and its corollaries occupies the first half of Section~\ref{s:nagpalshift}. Our approach largely parallels similar results for quotients of the polynomial ring $S$ from \cite[Section~5]{gan25pol}. Section~\ref{ss:technical} contains the new technical input, essentially a nonvanishing statement for $\Sh_{p^r}$ applied to certain socles of tensor products of $\GL$-representations (Lemma~\ref{lem:newlemma}).

In the rest of Section~\ref{s:nagpalshift}, we lay the groundwork for Theorem~\ref{thm:introsod} adapting an approach of Djament \cite[Appendix~A]{dj16lc}. In Section~\ref{ss:lcprelim}, we define local cohomology functors for $\Drs$-modules and prove that $\Sh_{p^r}$ commutes with taking local cohomology in $\Mod_{\Drs}$ (Proposition~\ref{prop:shcommutesderivedgamma}). Using this, we prove finite generation results for these functors (Proposition~\ref{prop:lcfinitenessDrs}).   

We finally prove Theorems~\ref{thm:gensdbmodintro} and~\ref{thm:introsod} in Section~\ref{s:fpd}. We first deduce a shift theorem for $\Dr$, i.e., given a finitely presented $\Dr$-module $M$, the $\Dr$-module $\Sh_{p^r}^t(M)$ is flat for $t \gg 0$ (Theorem~\ref{thm:shiftDr}), from Theorem~\ref{thm:shiftintro}. Theorem~\ref{thm:gensdbmodintro} then follows easily from the shift theorem. To prove Theorem~\ref{thm:introsod}, we first prove some results comparing local cohomology modules of a $\Dr$-module computed in $\Mod_{\Dr}$ and after restricting to $\Mod_{\Drs}$ for $s \gg 0$. Using this, we prove vanishing results for local cohomology of induced $\Dr$-modules. Theorem~\ref{thm:introsod} is now quite formal although the lack of noetherianity results makes the argument slightly nuanced.

\begin{remark}\label{rmk:DvsBinf}
    The subtleties of this paper can be perhaps best illustrated by the dependence of the results on the $\GL$-structure of the divided power algebra.
    
    Consider the subalgebra $D_{[0,1]} \subset D$; this is the $\GL$-subalgebra generated by elements of degree $\leq p$. Explicitly, it is the subalgebra generated by $x_i, x_i^{[p]}$ for all $i \in \bN_{>0}$. Also consider the algebra $B = D_{[0,0]} \otimes D_{[0,0]}^{(1)}$. These two algebras are isomorphic as $k$-algebras but not as $\GL$-algebras. It is not hard to show that $B$ is also $\GL$-noetherian but the similarities end there.
     
    The $\GL$-spectrum of $D_{[0,1]}$ has three points: the ideal $(0)$, the ideal $(x_i \,|\, i \in \bN_{>0})$ and the homogeneous maximal ideal $(x_i, x_i^{[p]} \,|\, i \in \bN_{>0})$; the $\GL$-spectrum of $B$ has at least four points: the ideal $(0)$, the ideal $(x_i \,|\, i \in \bN_{>0})$, the ideal $(x_i^{[p]} \,|\, i \in \bN_{>0})$, and the homogeneous maximal ideal $(x_i, x_i^{[p]} \,|\, i \in \bN_{>0})$. 
    
    The analogue of Theorem~\ref{thm:shiftintro} fails for $B$: we have a surjective map $\pi \colon B \to D_{[0,0]}$ and $\Sh_1^t(D_{[0,0]}) \cong D_{[0,0]}$ is not a flat $B$-module for all $t$. The map $\pi$ is obtained by killing the ideal generated by $x_i^{[p]}$. In $D_{[0,1]}$, the $\GL$-stable ideal $I$ generated by $x_i^{[p]}$ has the element $x_1 x_2 \cdots x_p$ so in particular killing $x_i^{[p]}$ will result in killing $x_1 x_2 \cdots x_p$ and in turn, the $p$-th power of the ideal generated by the $x_i$. From this, it is not too difficult to see that $\Sh_1^{p+1}(D_{[0,1]}/ I) = 0$.
\end{remark}

\subsection{Relation to other work}
\subsubsection{Determinantal ideals}
In their seminal work \cite{ss16gl}, Sam--Snowden developed the theory of $\Sym(\bV)$ over fields of characteristic zero. With this third installment, all four of its characteristic $p$ analogues have now been studied. In a sequel paper \cite{ss19gl2}, Sam--Snowden go on to study $\Sym(\bV^{\oplus r})$ for $r > 1$; these provide the natural setting to investigate the asymptotic behavior of determinantal ideals and $\GL$-equivariant sheaves on the Grassmannian. We have yet to carry this theory over to positive characteristic; even determining the $\GL$-spectrum of $\Sym(\bV^{\oplus r})$ appears to be a substantially harder problem.

\subsubsection{Koszul duality}
The resolutions of finitely generated $S$-modules also exhibit interesting finiteness properties. In characteristic zero, this is fully understood \cite[Section~6 and 7]{ss16gl} but only partially so in positive characteristic \cite{gan24lnnr}. In particular, we expect actions of Frobenius twisted exterior algebras $\lw(\bV)^{(r)}$ on the free resolutions as explained in the postlude of ibid., essentially since $\lw(\bV)^{(r)}$ acts on the Koszul complex which resolves $S/\fm^{[p^r]}$ where $\fm \subset S$ is the maximal ideal. In the final subsection, we sketch how the new ideas from this paper can be supplemented to the results of \cite{gan22ext} to obtain similar results for $\Mod_{\lw(\bV^{(r)})}$.

Similarly, we also expect finiteness results for resolutions of finitely generated modules over the quotient algebras $S/\fm^{[p^r]}$. The Tate resolution takes the role of the Koszul complex here, so we hope results from this paper will help towards this but have not pursued it yet.

\subsubsection{General affine groups}
In forthcoming work, we study $D$-modules supported at the maximal ideal $I_{\infty}$, i.e., locally finite-length $D$-modules. We show that $\Mod_D^{\lf}$, the category of locally finite-length $D$-modules, is equivalent to polynomial representations of the infinite-rank general affine group, and also equivalent to the category of generic $S$-modules. In particular, the latter category was described in \cite[Section~4]{gan25pol}. Such results fall under the framework of the representation theory of \textit{linear-oligomorphic groups}, a new class of groups introduced by Harman--Snowden \cite{hs24tensor}. In this forthcoming work, we also study locally finite-length $D^{\otimes n}$-modules for all $n \geq 1$.

\subsubsection{Finite-variable divided power algebras}
The module theory of finite-variable divided power algebras is not as well-studied as one would expect. Nagpal--Snowden \cite{ns17div} develop the module theory over (generalizations of) the one-variable divided power algebra. Our results can be used to answer Le--Nagel--Nguyen--R\"omer \cite{lnnr20pdim, lnnr21reg} type questions for compatible sequences $\{M_n\}$ where $M_n$ is a $\GL_n$-equivariant module over the divided power algebra $\Div(k^n)$, finitely presented in a suitable sense. 
These correspond to $\GL$-equivariant $D$-modules by taking the limit as $n \to \infty$; see \cite[Section~2.1]{gan24lnnr} for the formalism of Vec-algebras to make this precise. For example, Theorem~\ref{thm:gensdbmodintro} easily implies that the \textit{slope} of the resolution of such families is eventually constant. It would be interesting to prove similar results for $\fS_{\infty}$-equivariant sequences as well.

\subsubsection{Shift theorem}
The shift theorems of this paper (Theorems~\ref{thm:shiftintro} and~\ref{thm:shiftDr}) are the latest in a long line of such results, beginning with Nagpal's theorem for $\FI$-modules \cite{nag15fi}, its extension to $\VI$-modules \cite{nag19vi}, and including our own shift theorems for $\Mod_R$ and $\Mod_S$ \cite{gan22ext, gan25pol}. While these results share a common outline, each setting introduces its own technical complications and a unifying framework remains elusive.

\subsection{Notation}
\begin{description}[align=right,labelwidth=2.5cm,leftmargin=!]
\item[$k$] an algebraically closed field of characteristic $p > 0$
\item[$\bV$] the infinite dimensional $k$-vector space with basis $\{x_i\}_{i \ge 1}$
\item[$\GL$] the group of automorphisms of $\bV$ fixing all but finitely many of the $x_i$
\item[$\Rep^\pol(\GL)$] the category of polynomial representations of $\GL$
\item[$\Pol$] the category of strict polynomial functors over $k$
\item[$F\{V\}$] the evaluation of the polynomial functor $F$ to the vector space $V$
\item[$L_{\lambda}$] the irreducible polynomial representation of $\GL$ with highest weight $\lambda$
\item[$-^{<n}$] the submodule generated by elements of degree $< n$
\item[$-^{(s)}$] the $s$-fold Frobenius twist of a $\GL$-representation/polynomial functor
\item[$t_i(-)$] the degree of $\Tor_i^A(-, k)$; the algebra $A$ is suppressed in this notation as it is usually clear in context
\item[$\Sh_m$] the $m$-th Hasse--Schur derivative functor where $m \in \bN$
\item[$D$] the divided power algebra $\Div\{\bV\}$
\item[$\Drs$] the subalgebra of $\Dr$ generated by elements of degree $\leq p^s$ ($s$ can be infinite)
\item[$\Jrd$] the ideal generated by elements of degree $\leq p^d$ in $\Drs$
\item[$\Mod_A$] the category of ($\GL$-equivariant) $A$-modules
\item[{$\Mod_A[\fa^{\infty}]$}] the category of $A$-modules locally annihilated by $\fa$ where $\fa \subset A$ is a $\GL$-stable ideal.
\item[$\rD^b_{\fpre}(\Mod_A)$] the bounded derived categories of finitely presented $A$-modules (assuming $A$ is $\GL$-coherent)
\end{description}
All tensor products $\otimes$ are over $k$ unless denoted otherwise.

\subsection*{Acknowledgements}{The author thanks Steven Sam for helpful discussions and Kent Vashaw for helpful references.}

\subsection*{Disclosure of LLM use}{The author used Claude Opus 4.8 and Fable 5 to revise the paper; these models identified gaps in the proofs of some results and errors in some statements that substantially changed the mathematical meaning; all such issues were corrected by the author.}

\section{Generalities}\label{s:prelim}
We refer the reader to \cite[Section~2]{gan22ext} for details on polynomial representations of $\GL$ and strict polynomial functors relevant to the contents of this paper. We emphasize a nonstandard notation we use: given a polynomial functor $F$ and a $k$-vector space $V$, we let $F\{V\}$ be the result of applying $F$ to $V$. Similarly, given a $\GL$-representation $M$ and $k$-vector space $V$, we let $M\{V\}$ be the result of applying $M$, regarded as a polynomial functor, on $V$. Throughout this section, $A$ will denote an arbitrary $\GL$-algebra with $A_0=k$. 

\subsection{\texorpdfstring{$A$}{A}-modules}\label{ss:modules}
An $A$-module refers to a module object for $A$ in $\Rep^{\pol}(\GL)$. Spelled out, an $A$-module $M$ is the data of a $k$-vector space $M$ with an action of the $k$-algebra $A$ and a linear action of $\GL$ such that $M$ is a polynomial representation of $\GL$ satisfying $g(am) = g(a) g(m)$ for all $g \in \GL, a \in A$ and $m \in M$.

Given an $A$-module $M$ and a homogeneous element $m \in M$, we let $\langle m \rangle$ denote the $A$-submodule generated by $m$, i.e., the smallest $\GL$-stable $A$-submodule of $M$ containing the element $m$. We extend this notation in the obvious way when we have a collection of elements $\{m_{\alpha}\} \subset M$ as well. 

\begin{defn}
 An $A$-module $M$ is finitely presented if there exists an exact sequence $A \otimes W \to A \otimes U \to M \to 0$ where $W$ and $U$ are finite-length polynomial $\GL$-representations.
\end{defn}

We now state the fundamental definition.
\begin{defn}
    A $\GL$-algebra is $\GL$-noetherian if every submodule of a finitely generated $A$-module is also finitely generated.
    A $\GL$-algebra is $\GL$-coherent if every finitely generated submodule of a finitely presented $A$-module is also finitely presented.
\end{defn}

Given an $A$-module $M$, let $\Tor_i^A(M, -)$ be the left derived functors of the right exact functor $M \otimes_A -$. An $A$-module $M$ is \textit{flat} if the functor $M \otimes_A -$ is exact. We let $t_i(M) = \maxdeg \Tor_i^A(M, k)$ as a polynomial representation with the convention that $\maxdeg(0) = -1$; we suppress the algebra $A$ in this notation as it is usually clear which $\GL$-algebra we are working with.

\begin{lemma}\label{lem:fptor}
    An $A$-module $M$ is
    \begin{enumerate}
        \item finitely generated if and only if $\Tor_0^A(M, k)$ is a finite-length $\GL$-representation; 
        \item generated in degrees $\leq n$ if and only if $\Tor_0^A(M,k)$ is supported in degrees $\leq n$;
        \item finitely presented if and only if $\Tor_0^A(M, k)$ and $\Tor_1^A(M,k)$ are finite-length $\GL$-representations.
    \end{enumerate}
\end{lemma}
\begin{proof}
This is the graded Nakayama's lemma, given that for a finite length $W$ of degree $\leq n$ in $\Rep^{\pol}(\GL)$, there exists a finite-length projective object $P$ of degree $\leq n$ in $\Rep^{\pol}(\GL)$ such that $P$ maps onto $W$. 
\end{proof}
If $t_0(M) = n$, the quotient $M/M^{<n}$ is nonzero and generated in degree $n$; we implicitly use this throughout the paper.

The cokernel of a map between finitely presented modules is always finitely presented. With the $\GL$-coherence assumption, we also have:
\begin{proposition}
    Let $A$ be a $\GL$-coherent $\GL$-algebra. The category of finitely presented $A$-modules is an abelian subcategory of $\Mod_A$.
\end{proposition}
\begin{proof}
 Let $\phi \colon M \to N$ be a map between two finitely presented $A$-modules. The image $\im \phi$ is finitely generated and a submodule of $N$ so it is also finitely presented. In particular, the $\GL$-representations $\Tor_0^A(\im \phi, k), \Tor_0^A(M, k), \Tor_1^A(M, k)$ and $\Tor_1^A(\im \phi, k)$ are all finite length by Lemma~\ref{lem:fptor}. From the long exact sequence of $\Tor^A(-, k)$ associated to the short exact sequence $0 \to \ker \phi \to M \to \im \phi \to 0$, we get that $\Tor_0(\ker \phi, k)$ is also finite length. So $\ker \phi$ is finitely generated, but by $\GL$-coherence of $A$, it is also finitely presented.
\end{proof}

Let $M, N$ be $\GL$-representations and $W$ be a direct summand of $\bV$. A homogeneous element $m \in M$ \textit{belongs to} $M\{W\}$ if the map $M\{\bV\} \to M\{W\} \to M\{\bV\}$ acts by identity on $m$. Two homogeneous elements $m \in M$ and $n \in N$ are \textit{disjoint} if there exists a decomposition $\bV = W_1 \oplus W_2$ such that $m$ belongs to $M\{W_1\}$ and $n$ belongs to $N\{W_2\}$. 

The next lemma is \cite[Lemma~2.16]{gan25pol}; the corollary that follows is \cite[Corollary~2.17]{gan25pol}.

\begin{lemma}\label{lem:disjointnzd}
Let $M$ be an $A$-module. Assume $a \in A$ and $m\in M$ are disjoint elements such that $am=0$. Then $\langle a\rangle \langle m \rangle = 0$.
\end{lemma}

\begin{corollary}\label{cor:disjoint}
Let $M$ be an $A$-module. For disjoint elements $a \in A$ and $m\in M$, we have $\langle a m \rangle = \langle a \rangle \langle m \rangle$.
\end{corollary}

\subsection{Semi-induced modules} \label{ss:semiind}
We introduce an extremely important class of $A$-modules.

\begin{definition}
An $A$-module is \textit{induced} if it is isomorphic to $A \otimes W$ for some polynomial representation $W$. An $A$-module is \textit{semi-induced} if it has an ascending filtration where the successive quotients are induced modules.
\end{definition}

Among induced modules, modules of the form $A \otimes P$ where $P$ is a projective polynomial representation are the projective $A$-modules. 

The proofs of the next four results are in \cite[\S~2.4]{gan22ext}.
\begin{proposition}\label{prop:flatequalssemi}
    Let $M$ be a finitely generated $A$-module. The following are equivalent:
    \begin{enumerate}[label=(\alph*), wide=2.1em]
        \item $M$ is semi-induced,
        \item $M$ is flat,
        \item $\Tor_i^A(M, k) = 0$ for all $i > 0$, and
        \item $\Tor_1^A(M, k) = 0$.
    \end{enumerate}
\end{proposition}

\begin{lemma}\label{lem:relationsinlowdeg}
     Let $M$ be an $A$-module generated in degree $n$ such that $t_1(M) \leq n$. The natural map $A \otimes M_n \to M$ is an isomorphism.
\end{lemma}

\begin{corollary}\label{cor:semises}
Let $0 \to M_1 \to M_2 \to M_3 \to 0$ be a short exact sequence of finitely generated $A$-modules.
\begin{enumerate}[label=(\alph*),wide=2.1em]
    \item If $M_1$ and $M_3$ are semi-induced, then so is $M_2$.
    \item If $M_2$ and $M_3$ are semi-induced, then so is $M_1$.
\end{enumerate}
\end{corollary}


\begin{lemma} \label{lem:torcomparison}
Given $A$-modules $M$ and $N$, and a finite dimensional vector space $V$ over $k$, we have isomorphisms $[\Tor_i^A(M,N)]\{V\} \cong \Tor_i^{A\{V\}}(M\{V\},N\{V\})$ for all $i \geq 0$.
\end{lemma}

\begin{lemma}
Assume $A$ is $\GL$-coherent. Let $M$ be a finitely presented $A$-module. Then $M$ has a left resolution by finitely generated induced $A$-modules.
\end{lemma}
\begin{proof}
    We have a surjective map $F_0 \coloneqq A \otimes W \to M$ where $W \subset M$ is a finite-length $\GL$-representation that generates $M$. The module $F_0$ is induced and finitely generated, so finitely presented as well. So the kernel of this map between finitely presented modules is finitely presented as $A$ is $\GL$-coherent, so as in the classical case, we can continue building the resolution by replacing $M$ with $\ker(F_0 \to M)$.
\end{proof}

\begin{corollary}
    Assume $A$ is $\GL$-coherent. Let $M$ be a finitely presented $A$-module. The Tor groups $\Tor_i^A(M, k)$ are finite length $\GL$-representations for all $i$.
\end{corollary}
\begin{proof}
    By the previous lemma, the module $M$ has a left resolution by finitely generated induced $A$-modules which is a flat resolution by Proposition~\ref{prop:flatequalssemi}. Applying the functor $k \otimes_A (-)$ to this resolution results in a complex of finite-length $\GL$-representations so the homology is also finite-length. This homology is precisely the requisite $\Tor$ group.
\end{proof}

\subsection{\texorpdfstring{$\GL$}{GL}-prime ideals} \label{ss:glprimes}
An ideal of $A$ is a submodule of $A$, i.e., it is a usual ideal of the underlying $k$-algebra $A$ that is stable under the action of $\GL$.

\begin{defn}
A proper ideal $\fp \subset A$ is $\GL$-prime if for all ideals $\fa, \fb$ with $\fa \fb \subset \fp$, either $\fa$ or $\fb$ is contained in $\fp$. 
\end{defn}
A $\GL$-algebra is a \textit{$\GL$-domain} if the ideal $\langle 0 \rangle$ is a $\GL$-prime. The $\GL$-spectrum of $A$ is the set of $\GL$-prime ideals endowed with the Zariski topology. 

\subsection{Torsion modules}\label{ss:torsion}
Let $M$ be an $A$-module. A nonzero element $m \in M$ is $\textit{torsion}$ if there exists a nonzero ideal $I \subset A$ such that $Im = 0$.  We let $\Ann_A(M)$ be the sum of all ideals that annihilate $M$.

An $A$-module $M$ is \textit{torsion} if all nonzero elements of $M$ are torsion. The \textit{torsion submodule} $T(M)$ of $M$ is the submodule generated by all torsion elements. A module $M$ is \textit{torsion-free} if $T(M) = 0$.

\begin{lemma}\label{lem:minprimetors}
    Assume $A$ is a $\GL$-domain and $\fp$ is a nonzero $\GL$-prime ideal such that all nonzero $\GL$-stable ideals contain a power of $\fp$. An $A$-module is torsion if and only if it is locally annihilated by $\fp$.
\end{lemma}
\begin{proof}
The if direction is clear from definition. Assume an $A$-module $M$ is torsion. Then for all $m \in M$, there exists a nonzero $\GL$-stable ideal $I$ such that $Im = 0$. By assumption $\fp^t \subset I$ for $t \gg 0$, so $m$ is also annihilated by a power of $\fp$.
\end{proof}
In the setting of the above lemma, we may identify the subcategory of torsion modules with the subcategory of modules locally annihilated by the $\GL$-prime ideal of height one. The next definition generalizes the notion of a torsion-free module.
\begin{defn}
An $A$-module $M$ is \textit{prime} if for all nonzero submodules $N \subset M$, we have $\Ann_A(N) = \Ann_A(M)$.
\end{defn}

\begin{lemma}\label{lem:primeprime}
    Given a prime $A$-module $M$, its annihilator $\Ann_A(M)$ is $\GL$-prime.
\end{lemma}
\begin{proof}
    Let $I = \Ann_A(M)$. Given two $\GL$-stable ideals $\fa, \fb$ such that $\fa \fb \subset I$ with $\fa \not\subset I$, it suffices to show that $\fb$ is contained in $I$. The $A$-submodule $N_1 = \fa M$ is nonzero as $\fa$ is not contained in $I = \Ann_A(M)$. Furthermore, we have $\fb N_1 = \fb \fa M \subset IM = 0$, therefore $\fb \subset \Ann_A(N_1) = I$ as $M$ is prime, as required. 
\end{proof}

\begin{lemma}\label{lem:existenceass}
    Assume $A$ is $\GL$-noetherian and $M$ is a nonzero $A$-module. There exists a nonzero prime submodule $N \subset M$.
\end{lemma}
\begin{proof}
    Assume $M$ is not prime. Then there exists a submodule $N_1 \subset M$ such that $\Ann_A(M) \subsetneq \Ann_A(N_1)$. If $N_1$ is also not prime, we can find a submodule $N_2 \subset N_1$ such that $\Ann_A(N_1) \subsetneq \Ann_A(N_2)$. Continuing in this fashion, we get a descending chain of submodules of $M$ whose respective annihilators correspond to an ascending chain of ideals of $A$ which must stabilize as $A$ is $\GL$-noetherian. Therefore, the original descending chain must also terminate, and the submodule at which it terminates is a prime submodule, as required.
\end{proof}
The next result is a prime cyclic filtration result.
\begin{proposition}\label{prop:primecyclic}
   Assume $A$ is $\GL$-noetherian and $M$ is a finitely generated $A$-module. There is a finite filtration
   \[
   0 = F_0 \subset F_1 \subset \ldots F_{n-1} \subset F_n = M
   \]
   such that for all $0 < i \leq n$, the module $F_{i}/F_{i-1}$ is a torsion-free $A/\fp_i$-module for some $\GL$-prime ideal $\fp_i \subset A$.
\end{proposition}
\begin{proof}
Using the previous lemma and an easy noetherian induction, we see that $M$ has a finite filtration 
$0 = F_0 \subset F_1 \subset \ldots F_{n-1} \subset F_n = M$ such that for each $0 < i \leq n$, the module $F_i/F_{i-1}$ is a prime submodule of $M / F_{i-1}$. By Lemma~\ref{lem:primeprime}, the annihilator $\Ann_A(F_i/F_{i-1}) = \fp_i$ is a $\GL$-prime, so $F_i/F_{i-1}$ is an $A/\fp_i$-module and torsion-free as such because it is prime.
\end{proof}

\subsection{Coherence} \label{ss:coherence}
We state and prove our main criterion for $\GL$-coherence.
\begin{proposition}\label{prop:coherenceofflat}
For each $i \in \mathbf{N}_{>0}$, assume $A_i$ is a $\GL$-coherent $\GL$-algebra and there exists injective flat maps $\pi_i \colon A_i \to A_{i+1}$. Set $A = \colim A_i$.
The $\GL$-algebra $A$ is $\GL$-coherent.
\end{proposition}
\begin{proof}
Let $M$ be a finitely presented $A$-module and $N \subset M$ be a finitely generated submodule. Let $A \otimes W_2 \to A \otimes W_1 \to M \to 0$ be a finite presentation. Let $A \otimes U_1 \to N$ be a surjection from a finitely generated projective $A$-module to $N$. The inclusion $N \to M$ lifts to a map $A \otimes U_1 \to A \otimes W_1$. Choose $n \gg 0$ such that the image of both $U_1$ and $W_2$ in $A \otimes W_1$ is contained in $A_n \otimes W_1$. Let $T_M = \coker(A_n \otimes W_2 \to A_n \otimes W_1)$ and let $T_N$ be the image of $A_n \otimes U_1$ in $T_M$. The $A_n$-module $T_N$ is a finitely generated submodule of $T_M$, so it is also finitely presented. Let $A_n \otimes U_2 \to A_n \otimes U_1 \to T_N \to 0$ be a finite presentation of $T_N$. Applying the functor $A \otimes_{A_n}$, we obtain a finite presentation $A \otimes U_2 \to A \otimes U_1 \to A \otimes_{A_n} T_N \to 0$. By flatness, the module $A \otimes_{A_n} T_N = \im(A \otimes U_1 \to M) = N$, so $N$ is also finitely presented, as required.
\end{proof}

The submodule of elements of degree $> n$ for some fixed $n$ is rarely finitely presented; indeed this even fails for the divided power algebra and any positive $n$. We however have the following result:
\begin{lemma}
Let $A$ be a $\GL$-coherent $\GL$-algebra such that $A_i$ is a finite length $\GL$-representation for all $i \in \bN$. Let $M$ be a finitely presented $A$-module. For all $n \geq 1$, the $A$-submodule $M^{< n}$ generated by all elements of degree $< n$ is also finitely presented.
\end{lemma}
\begin{proof}
Since $M$ is finitely generated over $A$, there is a surjection $A \otimes W \to M$ with $W$ a finite length $\GL$-representation. The assumptions on $A$ imply that each graded piece of $A\otimes W$ is a finite length $\GL$-representation, and in turn, each graded piece of $M$ is also finite length. Therefore the submodule $M^{<n}$ is finitely generated, so also finitely presented by $\GL$-coherence of $A$.
\end{proof}
The next result does not require the coherence assumption so we omit its proof.
\begin{lemma}\label{lem:tensorfp}
    Let $M$ and $N$ be finitely presented $A$-modules. The tensor product $M \otimes_A N$ is also finitely presented.
\end{lemma}

\begin{lemma}\label{lem:IMfp}
    Let $A$ be a $\GL$-coherent $\GL$-algebra and $M$ be a finitely presented $A$-module. For a finitely generated ideal $I \subset A$, the submodule $IM \subset M$ is also finitely presented.
\end{lemma}
\begin{proof}
   The submodule $IM$ is finitely generated being the image of $I \otimes_A M$ which is finitely generated by the previous lemma, so $IM \subset M$ is also finitely presented because $A$ is  $\GL$-coherent.
\end{proof}


\subsection{Hasse--Schur derivatives}\label{ss:hasseschur} 
In this section, we recall a sequence of endofunctors of $\Rep^\pol(\GL)$ called the \textit{Hasse--Schur derivatives} which we introduced in \cite{gan25pol} to study $\Sym\{\bV\}$-modules. 

\subsubsection{Hasse--Schur derivatives of polynomial functors} Fix $m \geq 0$. Let $F \colon \Fec_k \to \Fec_k$ be a polynomial functor. We define the \textit{$m$-th Hasse--Schur derivative} of $F$, denoted $\Sh_m(F)$, to be the polynomial functor which on objects is defined by
\begin{displaymath}
\Sh_m(F)\{V\} = F\{k \oplus V\}^{[m]},
\end{displaymath}
where the superscript here denotes the subspace of $F\{k \oplus V\}$ on which the top one-dimensional torus $\bG_m$ acts with weight $m$. Given a linear map $h \colon V \to W$, the induced map $\Sh_m(F)(h)$ is given by restricting the map $F(\id_k \oplus h) \colon F\{k \oplus V\} \to F\{k \oplus W\}$ to the appropriate subspace. It is easy to see that $\Sh_m(F)$ is also a polynomial functor and that the assignment $F \to \Sh_m(F)$ is functorial. 

\subsubsection{Basic properties of the Hasse--Schur derivative}
For all $m$, the $m$-th Hasse--Schur derivative is an exact $k$-linear functor (because representations of $\bG_m$ are semisimple). It preserves finite length objects of $\Pol$. The sequence $\{\Sh_m\}$  satisfies the generalized Leibniz rule, i.e., $\Sh_m(F\otimes G) \cong \bigoplus_{i+j = m} \Sh_i(F) \otimes \Sh_j(G)$. Furthermore, since all the weights occurring in an $r$-fold Frobenius twisted representation are divisible by $p^r$, we also see that $\Sh_m(W^{(l)}) = 0$ for $0 < m < p^l$.

\subsubsection{Hasse--Schur derivatives of a polynomial representation}Since $\Pol$ and $\Rep^\pol(\GL)$ are equivalent abelian categories, we obtain a sequence of endofunctors of $\Rep^\pol(\GL)$, which we again denote by $\{\Sh_m \}$. For a polynomial representation $W$, we identify $\Sh_m(W)$ as a subspace of $W\{k \oplus \bV\}$, where $\bG_m$ acts with weight $m$. The newly introduced basis vector of $k \oplus \bV$ will usually be denoted by $f$ or $y$, possibly with subscripts on these variables, depending on the situation.

\subsubsection{Hasse--Schur derivative of an \texorpdfstring{$A$}{A}-module} Let $A$ be a $\GL$-algebra, and $M$ be an $A$-module. For each $m > 0$, the Hasse--Schur derivative $\Sh_m(M)$ is canonically an $A$-module, as we now explain. We have a $\GL$-equivariant map $A \to A\{k \oplus \bV\}$ induced by the canonical inclusion $\bV \to k \oplus \bV$. The $\GL$-representation $M\{k \oplus \bV\}$ is a $\GL(k \oplus \bV)$-equivariant module over $A\{k \oplus \bV\}$, so by restriction of scalars, it is a $\GL$-equivariant module over $A$. The subspace $\Sh_m(M)$ of $M\{k \oplus \bV\}$ is stable under the action of $\GL$ as well as multiplication by elements of $A$, so $\Sh_m(M)$ is an $A$-module.

\section{The algebra \texorpdfstring{$D$}{D}} \label{s:dmodules}
For the remainder of this paper, we let $D$ be the $\GL$-algebra $\Div\{\bV\}$. This algebra is generated over $k$ by elements of the form $x_i^{[p^r]}$ with $i \in \bN_{>0}$ and $r \in \bN$ each of which satisfies $(x_i^{[p^r]})^p = 0$. In particular, the $\GL$-algebra $D$ is not finitely generated.

For each integer $r \geq 0$, we let $\Drinfty$ be the $r$-fold Frobenius twist of $D$, and for $r \leq s \leq \infty$, we let $\Drs$ be the subalgebra of $\Drinfty$ generated by elements of degree $\leq p^s$, i.e., by the elements $x_i^{[p^t]}$ with $r \leq t \leq s$. We identify $\Drinfty$ with the quotient of $D$ by the $\GL$-stable ideal generated by $x_i^{[p^t]}$ with $t < r$. We therefore have inclusions $\Drs \to \Drsp$ and surjections $\Drs \to D_{[r+1,s]}$ of $\GL$-algebras. For $s < \infty$, the Frobenius twist of the $\GL$-algebra $\Drs$ is isomorphic to $D_{[r+1, s+1]}$.

\subsection{Proof of Theorem~\ref{thm:coherence}}\label{ss:coherenceD}
\begin{lemma}
    For $s<\infty$, the $\GL$-algebra $\Drs$ is $\GL$-noetherian.
\end{lemma}
\begin{proof}
Consider the subgroup of permutation matrices $\fS_{\infty} \subset \GL$. Restricting the $\GL$-action to $\fS_{\infty}$ and forgetting the grading, the algebra $\Drs$ is a quotient of $\Sym(\bV^{\oplus s - r + 1})$, which is $\fS_{\infty}$-noetherian by Cohen's theorem \cite{co67laws} (see also \cite{nr19fioi}). Since a polynomial representation of $\GL$ is finite-length if and only if it is finite-length as an $\fS_{\infty}$-representation, the result follows.
\end{proof}

Theorem~\ref{thm:coherence} is the $r=0$ case of the following result as $D = D_{[0, \infty]}$. 
\begin{proposition}\label{prop:coherence}
   The $\GL$-algebra $\Drinfty$ is $\GL$-coherent. 
\end{proposition}
\begin{proof}
   Clearly, $\Drinfty = \colim_{s \in \mathbf{N}} \Drs$. The map $\Drs \to \Drsp$ is flat: after forgetting the $\GL$-action, the algebra $\Drsp$ is obtained by adding an extra set of variables to $\Drs$ and killing their $p$-th power so $\Drsp$ is a free $\Drs$-module. Proposition~\ref{prop:coherenceofflat} now applies.
\end{proof}

We note a useful result about finitely presented $\Drinfty$-modules.
\begin{lemma} \label{lem:finitebase}
    Let $M$ be a finitely presented $\Drinfty$-module. There exists $s<\infty$ and a finitely presented $\Drs$-module $N$ such that $\Drinfty \otimes_{\Drs} N \cong M$.
\end{lemma}
\begin{proof}
Let $\Drinfty \otimes U \to \Drinfty \otimes W \to M \to 0$ be a finite presentation of $M$ (so $U$ and $W$ are finite length $\GL$-representations). Since $\Drinfty = \cup_{s \in \bN} \Drs$, the image of $U$ in $\Drinfty \otimes W$ is contained in $\Drs \otimes W$ for some $s < \infty$. Let $N = \coker(\Drs \otimes U \to \Drs \otimes W)$. The base change of $N$ to $\Drinfty$ is isomorphic to $M$ by right exactness of $\Drinfty \otimes_{\Drs}$.
\end{proof}

We will implicitly use this next result throughout the paper.

\begin{corollary}
    Assume $M$ is a finitely presented $\Drinfty$-module annihilated by a finitely generated ideal $I$. Then $M$ is also a finitely presented $\Drinfty/I$-module.
\end{corollary}
\begin{proof}
    Assume $s$ is sufficiently large such that we have an ideal $J \subset \Drs$ and a finitely presented $\Drs$-module $N$ such that $J \Drinfty = I$ and $\Drinfty \otimes_{\Drs} N \cong M$. Clearly, the module $N$ is a finitely presented $\Drs/J$-module by $\GL$-noetherianity of $\Drs/J$; since the base change $\Drinfty/I \otimes_{\Drs/J} N \cong M$, we get that $M$ is finitely presented over $\Drinfty/I$ as well.
\end{proof}

\subsection{Proof of Theorem~\ref{thm:introspec}}\label{ss:glspec}
We now analyze the $\GL$-spectrum of the algebras introduced above. 

\begin{defn}
Let $0 \leq r \leq s \leq \infty$ with $r$ an integer. For $d$ an integer satisfying $ r \leq d  \leq s $, we define $\Jrd$ to be the ideal in $\Drs$ generated by $x_1^{[p^t]}$ with $r \leq t \leq d$. We also define $\Jrinfty$ to be the ideal in $\Drinfty$ generated by $x_1^{[p^t]}$ with $t \geq r$. 
\end{defn}

The ideal $\Jrinfty$ is the homogeneous maximal ideal in $\Drinfty$ and so it is not finitely generated. A word of caution: when we write $\Jrd$, we purposefully don't indicate the algebra we are working in but assume that $r \leq d \leq s$ since the extension of $\Jrd \subset \Drs$ in $D_{[r, t]}$ is $\Jrd \subset D_{[r, t]}$ and similarly for $r \leq d \leq s < t$, the contraction of $\Jrd \subset D_{[r, t]}$ to $\Drs$ is $\Jrd \subset \Drs$.

\begin{lemma}
    The $\GL$-algebra $\Drs$ is a $\GL$-domain. 
\end{lemma}
\begin{proof}
    Any ideal $I \subset D$ is monomial since it is stable under the action of the diagonal matrices inside $\GL$. Given nonzero ideals $\fa, \fb$ in $D$, choose monomials $m \in \fa$ and $n \in \fb$. Applying a permutation, we may assume they are disjoint. It is easy to see that $mn \ne 0$ whence $\fa \fb \ne 0$. So $D$ is a $\GL$-domain and since $\Drinfty$ is a Frobenius twist of $D$, it is also a $\GL$-domain for all $r > 0$. Now, for $r < s < \infty$, the $\GL$-algebra $\Drs$ is a subalgebra of $\Drinfty$ so these algebras are also $\GL$-domains.
\end{proof}

We have isomorphisms $\Drs/J_{[r,s]} \cong k$ and for $r \leq d < s$, isomorphisms $\Drs/\Jrd \cong D_{[d+1, s]}$, so the ideals $\Jrd \subset \Drs$ defined above are all $\GL$-prime.

\begin{lemma}\label{lem:socle}
    Let $\nu_n$ be the unique $p$-restricted partition of $n$ with all but one part equal to $p-1$, i.e., writing $n = a (p-1) + b$ with $b < p-1$, we have $\nu_n = ((p-1)^a, b)$.
    The socle of $\Div_n\{\bV\}$ is the irreducible representation with highest weight $\nu_n$. The socle of $\Div_n\{\bV\}^{(l)}$ is the irreducible representation with highest weight $p^l \nu_n$.
\end{lemma}
\begin{proof}
    The dual statement of the first claim, that the head of $\Sym_n\{\bV\}$ is the irreducible representation $L_{\nu_n}$, is \cite[Corollary~2.10]{perl26ide}. The second claim follows from the first claim
    since the mapping $(W \subset V) \mapsto (W^{(1)} \subset V^{(1)})$ provides an isomorphism of the subrepresentation lattice of a polynomial $\GL$-representation $V$ to its Frobenius twist.
\end{proof}

\begin{lemma}\label{lem:Dideals}
    Assume $I \subset \Drs$ is a nonzero ideal. Then $I$ contains $\langle x_1^{[p^r]} \rangle^t$ for $t \gg 0$. 
\end{lemma}
\begin{proof}
Assume $I$ is nonzero in degree $p^r t$. Since $I_{p^rt} \subset (\Drs)_t \subset \Dr_t$, the ideal contains the simple socle of $\Div_t\{\bV\}^{(r)}$, which in particular contains the element $x_1^{[p^r]} x_2^{[p^r]} \cdots x_t^{[p^r]}$. The product of these elements generates the ideal $ \langle x_1^{[p^r]} \rangle^t$ by Corollary~\ref{cor:disjoint}.
\end{proof}

Theorem~\ref{thm:introspec} is the $r = 0, s=\infty$ case of the next proposition (that the map defined in the theorem statement is a homeomorphism is obvious).

\begin{proposition}\label{prop:spectrum}
Let $I \subset \Drs$ be a nonzero $\GL$-prime ideal. Then $I = \Jrd$ for some $r \leq d \leq s$.
\end{proposition}
\begin{proof}
   By Lemma~\ref{lem:Dideals}, the ideal $I$ contains some power of the ideal $\Jrr$ so it contains $\Jrr$ by primality. Let $d = \sup\{t \,|\, J_{[r,t]} \subset I \}$ -- this set is nonempty so either $d \geq r$ is an integer or $d = \infty$. 
   If $d = s$, then $I = J_{[r,s]}$ is the maximal ideal. If $d < s$, then we may go modulo $\Jrd$, and the image $\overline{I}$ of $I$ is a $\GL$-prime in $\Drs/\Jrd \cong D_{[d+1, s]}$. If $\overline{I} \ne 0$, then again applying Lemma~\ref{lem:Dideals}, we see that $J_{[d+1, d+1]} \subset \overline{I}$, and in turn, $J_{[r, d+1]} \subset I$, which contradicts our choice of $d$. Thus we have $I = \Jrd$ for some $d$ satisfying $r \leq d \leq s$, as required. 
\end{proof}

\subsection{Torsion \texorpdfstring{$D$}{D}-modules}\label{ss:torsionD}
We now prove some basic results about torsion $\Drs$-modules; this class was introduced for arbitrary $\GL$-algebras in Section~\ref{ss:torsion}.

\begin{lemma}\label{lem:torsionJrr} 
    Assume $M$ is a finitely generated torsion $\Drs$-module. Then $\Jrr^n M = 0$ for $n \gg 0$.
\end{lemma}
\begin{proof}
By Lemma~\ref{lem:Dideals}, every nonzero $\GL$-stable ideal contains a power of $\Jrr$, so by Lemma~\ref{lem:minprimetors}, a module is torsion if and only if it is locally annihilated by $\Jrr$. Let $m_1, m_2, \ldots, m_t$ be generators of $M$, and choose $e \gg 0$ such that $\Jrr^n m_i = 0$ for $i \in [t]$ and $n > e$. Clearly, for all $n > e$, we also have $\Jrr^n M =0$, as required.
\end{proof}

\begin{corollary}\label{cor:locannihilator}
    Assume a nonzero $\Drs$-module $M$ is not torsion-free. Then there exists a nonzero element $m \in M$ such that $\Jrr m = 0$. 
\end{corollary}
\begin{proof}
    First, assume $M$ is finitely generated. By the previous lemma, there exists $n > 0$ such that $\Jrr^n M = 0$. Choose the minimal such $n$ and take any nonzero element $m \in \Jrr^{n-1} M \ne 0$; clearly $\Jrr m = 0$. For arbitrary $M$, let $x$ be a nonzero torsion element and let $N$ be the $\Drs$-module generated by $x$. The module $N$ is not torsion-free and is finitely generated, so by the previous paragraph, we see that there exists a nonzero element $m \in N$ such that $\Jrr m = 0$.
\end{proof}

\begin{lemma} \label{lem:tfext}
Assume $M$ is a torsion-free $\Drs$-module. Then $\Drinfty \otimes_{\Drs} M$ is a torsion-free $\Drinfty$-module.
\end{lemma}
\begin{proof}
   Assume for the sake of contradiction that the $\Drinfty$-module $\Drinfty \otimes_{\Drs} M$ is not torsion-free. By Corollary~\ref{cor:locannihilator}, then there exists a nonzero element $x$ such that $\Jrr \subset \Drinfty$ annihilates $x$. By restriction of scalars, we see that the $\Drs$-module $\Drinfty \otimes_{\Drs} M$ is also not torsion-free as $x$ is also annihilated by $\Jrr \subset \Drs$. However, forgetting the $\GL$-action, the $\Drs$-module $\Drinfty$ is free, so $\Drinfty \otimes_{\Drs} M$ is a direct sum of copies of $M$ thus also torsion-free, a contradiction.
\end{proof}
\begin{lemma} \label{lem:torsext}
Assume $M$ is a torsion $\Drs$-module. Then $\Drinfty \otimes_{\Drs} M$ is a torsion $\Drinfty$-module.
\end{lemma}
\begin{proof}
Since $M$ is torsion, every element $m \in M$ is killed by a sufficiently large power $\Jrr^{\tau(m)}$ by Lemma~\ref{lem:torsionJrr}. Let $x = \sum_{i=1}^n d_i \otimes m_i \in \Drinfty \otimes_{\Drs} M$. Take an integer $t > \max\{\tau(m_1), \tau(m_2), \ldots, \tau(m_n)\}$. Then $\Jrr^t x = 0$, i.e., $x$ is torsion. Since $x \in 
\Drinfty \otimes_{\Drs} M$ was arbitrary, we obtain the result.
\end{proof}
\begin{lemma}\label{lem:extension}
 The base change of the $\Drs$-module $\Drs/\Jrd \otimes L_{\lambda}$ along the extension $\Drs \to \Drinfty$ is isomorphic to $D_{[d+1, \infty]} \otimes L_{\lambda}$. 
\end{lemma}
\begin{proof}
   The result follows from the chain $ \Drinfty \otimes_{\Drs} (\Drs/\Jrd \otimes L_{\lambda}) \cong \Drinfty/\Jrd \otimes L_{\lambda} \cong D_{[d+1, \infty]} \otimes L_{\lambda}$.
\end{proof}

\begin{proposition} \label{prop:torsionfp}
    The torsion submodule of a finitely presented $\Drinfty$-module is also finitely presented.
\end{proposition} 
\begin{proof}
   Let $M$ be a finitely presented $\Drinfty$-module. Let $s< \infty$ and $N$ be a finitely generated $\Drs$-module such that $\Drinfty \otimes_{\Drs} N \cong M$; such an $s$ and $N$ exist by Lemma~\ref{lem:finitebase}. Let $\mGaq(N)$ be the torsion submodule of $N$; this is finitely generated by $\GL$-noetherianity of $\Drs$. Tensoring by $\Drinfty$ the short exact sequence $0 \to \mGaq(N) \to N \to N/\mGaq(N) \to 0$, we get a short exact sequence 
   \[
    0 \to \Drinfty \otimes_{\Drs} \mGaq(N) \to M \to \Drinfty \otimes_{\Drs} N/\mGaq(N) \to 0
   \]
   by flatness of the extension $\Drs \to \Drinfty$. By Lemma~\ref{lem:tfext}, the module 
   $\Drinfty \otimes_{\Drs} N/\mGaq(N)$ is torsion-free, so 
   $\Drinfty \otimes_{\Drs} \mGaq(N)$ contains the torsion submodule of $M$. Furthermore, the $\Drinfty$-module $\Drinfty \otimes_{\Drs} \mGaq(N)$ is torsion by Lemma~\ref{lem:torsext}, so it is contained in $\mGaq(M)$, hence $\mGaq(M) = \Drinfty \otimes_{\Drs} \mGaq(N)$, whence we see that $\mGaq(M)$ is also finitely presented.
\end{proof}
\begin{remark}
Our proof of Proposition~\ref{prop:torsionfp} crucially uses the fact that $D$ is a colimit of $\GL$-noetherian $\GL$-algebras. The analogous result holds for finitely presented modules over a (usual) coherent $k$-algebra; note however that there are coherent $k$-algebras which are not flat colimits of noetherian $k$-algebras. We do not know how to prove the above proposition for arbitrary $\GL$-coherent $\GL$-algebras (nor do we know of an example of a $\GL$-coherent $\GL$-algebra that is not a colimit of $\GL$-noetherian ones).
\end{remark}

We end with a prime-cyclic filtration result for $\Drs$.
\begin{proposition}\label{prop:devissage}
   Let $M$ be a finitely presented $\Drs$-module. There is a finite filtration
   \[
   0 = F_0 \subset F_1 \subset \ldots F_{n-1} \subset F_n = M
   \]
   such that for all $i <n$, the module $F_{i+1}/F_i$ is a torsion-free finitely presented $\Drs/\fp_i$-module with $\fp_i \in \Spec_{\GL}(\Drs)$ a finitely generated ideal. 
\end{proposition}
\begin{proof}
    For $s < \infty$, this is just Proposition~\ref{prop:primecyclic} combined with the computation of the $\GL$-spectrum (Proposition~\ref{prop:spectrum}). 

    Now let $s = \infty$. By Lemma~\ref{lem:finitebase}, there exists $d < \infty$ and a finitely presented $D_{[r, d]}$-module $N$ with $\Drinfty \otimes_{D_{[r, d]}} N \cong M$. Let $0 = F_0 \subset \cdots \subset F_n = N$ be a filtration of $N$ satisfying the conclusion of the proposition. Since $D_{[r,d]} \to \Drinfty$ is flat, the modules $\Drinfty \otimes_{D_{[r,d]}} F_i$ form a filtration of $M$ with graded pieces $\Drinfty \otimes_{D_{[r,d]}} (F_{i+1}/F_i)$, each annihilated by $\fp_i \subset \Drinfty$, so a $\Drinfty/\fp_i$-module. They are finitely presented since base change preserves finite presentation and torsion-free by Lemma~\ref{lem:tfext}. By Proposition~\ref{prop:spectrum}, each $\fp_i$ is either the zero ideal or $J_{[r, e]}$ for some $r \leq e \leq d$. 
\end{proof}

\begin{lemma}\label{lem:torsiondevissage}
    A finitely presented torsion $\Drs$-module has a finite filtration where the successive quotients are finitely presented modules over $\Drs/\Jrr$.
\end{lemma}
\begin{proof}
    Let $M$ be a finitely presented torsion $\Drs$-module. By Lemma~\ref{lem:torsionJrr}, the module is annihilated by $\Jrr^n$ for some $n > 0$. The filtration $M \supset \Jrr M \supset \Jrr^2 M \ldots \supset \Jrr^{n-1} M \supset \Jrr^n M = (0)$ has the requisite property (note that $\Jrr^i M$ is finitely presented by Lemma~\ref{lem:IMfp} for all $i$).
\end{proof}

\subsection{The shift functor}\label{ss:shiftdefn}
Recall from Section~\ref{ss:hasseschur} that for each integer $u$, the Hasse--Schur derivative $\mShu$ defines an endofunctor of $\Mod_{\Drs}$. 

For $q = p^r$, we define a natural map $\id_{\Mod_{\Drs}} \to \mShq$.
We have a $\Drs$-module map 
\begin{displaymath}
m \mapsto y_1^{[q]} m
\end{displaymath}
from $M$ to $M\{k \oplus \bV\}$, where $y_1$ is the new basis vector in $k \oplus \bV$ and $y_1^{[q]}$ is its $q$-th divided power. Since $m \in M\{\bV\}$, the element $y_1^{[q]}$ lies in the weight-$q$ subspace of the top $\bG_m \subset \GL(k \oplus \bV)$-action. So the map factors $M \to \mShq(M) \to M\{k \oplus \bV\}$. We let $i_M$ be the first map. It is easy to see that the maps $i_M$ define a natural transformation $i$.

We let $\mDeq = \coker(i)$ and $\mKq = \ker(i)$. It follows from the snake lemma that $\mDeq$ is a right exact functor and $\mKq$ is left exact.

We first prove some basic properties of the $q$-th Hasse--Schur derivative endofunctor on $\Mod_{\Drs}$ (which hereafter we refer to as the \textit{shift functor}). We also let $\mGaq$ be the functor that takes a $\Drs$-module to its torsion submodule. By Lemma~\ref{lem:torsionJrr}, we may identify $\mGaq$ with the functor that sends a module $M$ to its submodule containing elements locally annihilated by $\Jrr$.
\begin{proposition}\label{prop:basicshift}
    Let $M$ be a $\Drs$-module.
    \begin{enumerate}[label=(\alph*), wide=2.1em]
        \item The kernel $\mKq(M)$ of $i_M$ is contained in the torsion submodule $\mGaq(M)$ of $M$.
        \item The kernel $\mKq(M) = 0$ if and only if $M$ is torsion-free.
        \item If $M$ is torsion-free, then so is $\mShq(M)$.
        \item If $M$ is nonzero, the module $\mDeq(M)$ is generated in degrees $\le \max(-1,t_0(M) - q)$.
        \item If $M$ is a finitely generated semi-induced module, then so are $\mShq(M)$ and $\mDeq(M)$.
        \item If $M$ is finitely generated (resp.~finitely presented), then so is $\mShq(M)$ and $\mDeq(M)$.
        \item The functors $\mShq$ and $\mDeq$ commute, i.e., $\mShq \mDeq$ is naturally isomorphic to $\mDeq\mShq$.
        \item The functors $\mShq$ and $\mGaq$ commute.
        \item If $M$ is finitely presented, then $\mShq^t(M)$ is torsion-free for sufficiently large $t$.
    \end{enumerate}
\end{proposition}

The proposition will be proved at the end of this section, after we prove a few preliminary results. 

\begin{lemma}\label{lem:obvshiftlemma}
    Consider the two natural maps $\mShq \to \mShq^2$ given by $\mShq (i)$ and $i_{\mShq}$. There exists an involution $\tau$ of $\mShq^2$ such that $\mShq(i)=\tau \circ i_{\mShq}$.
\end{lemma}
\begin{proof}
The proof of \cite[Lemma~4.2]{gan22ext} applies mutatis mutandis.
\end{proof}

\begin{lemma} \label{lem:shiftofinduced}
    Let $V$ be a polynomial representation of $\GL$. We have isomorphisms
   \begin{displaymath}
   \mShq(\Drs \otimes V) \cong (\Drs \otimes V) \oplus (\Drs \otimes \Sh_q(V))
   \end{displaymath}
   Furthermore, the natural map $i_{\Drs \otimes V}$ is the identity map onto the first summand. Therefore, we have isomorphisms
  \begin{displaymath}
   \mDeq(\Drs \otimes V) \cong \Drs \otimes \Sh_q(V)
  \end{displaymath} 
\end{lemma}
\begin{proof}
We have $\Sh_m(\Drs) = 0 $ for $0 < m < q$ and $\mShq(\Drs) \cong \Drs$. Thus we obtain the result using the generalized Leibniz rule (see also the proof of \cite[Lemma~4.3]{gan22ext}).
\end{proof}
\begin{corollary}\label{cor:dfrobzero}
    Let $W$ be a polynomial representation of $\GL$ with $\mShq(W) = 0$. Then $\mDeq(\Drs\otimes W) = 0$.
\end{corollary}
\begin{proof}
This follows by using the last part of Lemma~\ref{lem:shiftofinduced}.
\end{proof}
\begin{proof}[Proof of Proposition~\ref{prop:basicshift}] 
\leavevmode
\begin{enumerate}[label=(\alph*), wide=2.1em]
    \item  Assume $i_M(m) = y_1^{[q]}m = 0$ in $\mShq(M)$ for some nonzero $m \in M$. As the module $\mShq(M)$ is contained in $M\{k \oplus \bV\}$, we have $y_1^{[q]}m = 0$ in $M\{k \oplus \bV\}$, which is a $\GL(k \oplus \bV)$-equivariant module over $\Drs\{k \oplus \bV\}$. Since $y_1^{[q]}$ and $m$ are disjoint, we get that $\langle y_1^{[q]} \rangle \langle m \rangle = 0$ by Lemma~\ref{lem:disjointnzd}. In particular, we get that the ideal $\langle x_1^{[q]} \rangle \subset \Drs$ annihilates $m \in M$, and so $m$ is torsion.
    \item If $M$ is torsion-free, then by the previous part, we have $\ker(i_M)=0$. Now assume $M$ is not torsion-free. By Corollary~\ref{cor:locannihilator}, we may choose a nonzero element $m$ annihilated by the ideal $\langle x_1^{[q]} \rangle$. Choose $t \gg 0$ such that $x_t^{[q]}$ and $m$ are disjoint. Then $x_t^{[q]} m = 0$. The group $\GL(k \oplus \bV)$ acts on $\Drs\{k \oplus \bV\}$, so we may apply the element $g$ of $\GL(k \oplus \bV)$ which sends $x_t \to x_t + y_1$ and fixes every other basis vector to the equation $x_t^{[q]} m = 0$ to obtain $\sum_{i=0}^q x_t^{[i]} y_1^{[q-i]} m = 0$. But in $\Drs$, the $j$-th divided powers for all $0 < j < q$ vanish, so the equation simplifies to $x_t^{[q]} m + y_1^{[q]} m = 0$ which in turn implies that $y_1^{[q]} m =0$, or $m \in \ker(i_M)$ as required.
    \item If $M$ is torsion-free, then by part (a), the map $i_M$ is injective. Applying the exact functor $\mShq$, we see that the map $\mShq(i_M) \colon \mShq(M) \to \mShq^2(M)$ is also injective. By Lemma~\ref{lem:obvshiftlemma}, we have $i_{\mShq(M)} = \tau_M \circ \mShq(i_M)$, where $\tau_M$ is an involution of $\mShq^2(M)$. Therefore, the map $i_{\mShq(M)}$ is also injective and the module $\mShq(M)$ is torsion-free by part (b).
    \item  For a polynomial representation $W$ of degree $\le n$, the representation $\mShq(W)$ has degree $\le \max\{n-q,-1\}$.  Therefore by Lemma~\ref{lem:shiftofinduced}, applying $\mDeq$ to an induced module generated in degree $\le n$ results in an induced module generated in degree $\le n-q$. For an $\Drs$-module $M$, we have a surjection $\Drs\otimes W \to M$, with $W$ a polynomial representation of degree $\le t_0(M)$. Since $\mDeq$ is right exact, we have a surjection $\Drs \otimes \mShq(W) \to \mDeq(M)$, which implies that $t_0(\mDeq(M)) \le t_0(M) - q$, as required. 
    \item This follows from Lemma~\ref{lem:shiftofinduced}, and induction on the length of the filtration of $M$.
    \item This also follows from Lemma~\ref{lem:shiftofinduced}, exactness of $\mShq$, right exactness of $\mDeq$, and the fact that the Hasse--Schur derivative of a finite length polynomial representation is also finite length. 
    \item The module $\mShq(\mDeq(M))$ is the cokernel of $\mShq(i_M)$, and $\mDeq(\mShq(M))$ is the cokernel of $i_{\mShq(M)}$. These two maps only differ by an automorphism of the target by Lemma~\ref{lem:obvshiftlemma} so their cokernels are isomorphic.
    \item We have the short exact sequence $0 \to \mGaq(M) \to M \to M/\mGaq(M) \to 0$. Applying the functor $\mShq$ to this short exact sequence, we get 
    \begin{displaymath}
    0 \to \mShq(\mGaq(M)) \to \mShq(M) \to \mShq(M/\mGaq(M)) \to 0.
    \end{displaymath}
    By part (c), the module $\mShq(M/\mGaq(M))$ is torsion-free, so the submodule $\mGaq(\mShq(M))$ is contained in the image of $\mShq(\mGaq(M))$. For the reverse containment, the image of $\mShq(\mGaq(M))$ is a torsion submodule of $\mShq(M)$, and so is contained in $\mGaq(\mShq(M))$.
    \item First, assume $M$ is annihilated by $\Jrr$ so that it is a finitely generated $\Drs/\Jrr$-module. We have a surjection $\Drs/\Jrr \otimes V \to M$ with $V$ a $\GL$-representation of degree $\leq t_0(M)$. The $\GL$-representation $\Drs/\Jrr$ has no weight vectors with any component $= q$, so $\mShq(\Drs/\Jrr) = 0$. Combining this with the generalized Leibniz rule, we get $\mShq(\Drs/\Jrr \otimes V) = \Drs/\Jrr \otimes W$ with $W$ a $\GL$-representation of degree $\leq t_0(M) -q$. So we see that $\mShq^t(M)$ vanishes for $t > t_0(M)/q + 1$. Given an arbitrary finitely generated torsion $\Drs$-module $M$, it has a finite filtration by $\Drs/\Jrr$-modules by Lemma~\ref{lem:torsiondevissage}. We have shown that $\mShq^t$ annihilates the successive quotients of $M$, so by exactness of $\mShq$, we see that $\mShq^t(M)$ also vanishes for $t \gg 0$. The result now follows by part (h) and Proposition~\ref{prop:torsionfp}.
\end{enumerate} 
\end{proof}

\section{Finitely generated subalgebras of \texorpdfstring{$D$}{D}}\label{s:nagpalshift}
Throughout this section, we assume $r \leq s < \infty$ and set $q = p^r$. We study the structure of modules over the $\GL$-noetherian $\GL$-algebra $\Drs$.

\subsection{Some results on weights}\label{ss:technical}
The goal of this section is to prove Lemma~\ref{lem:newlemma} which will only be used in the proof of Proposition~\ref{prop:dmzerogenonedeg}. We recall some notation from \cite[Section~4]{gan25pol}.

For each integer $t > 0$, we let $G(t) \subset \GL$ be the subgroup containing block matrices of the form
$$\begin{pmatrix}
    \id_t & 0 \\
    0 & * 
\end{pmatrix}$$
where the top block is the identity matrix of size $t \times t$. 
The subgroup $G(t)$ is isomorphic to $\GL$ with isomorphism given by $A \mapsto \begin{pmatrix}
    \id_t & 0 \\ 0 & A
\end{pmatrix}$. 
Given a polynomial representation $W$, we let $\nSh_t(W)$ be the $\GL$-representation  with the same underlying space as $W$ but $\GL$ acting via the self-embedding $\GL \cong G(t) \subset \GL$ above. 

\begin{lemma}\label{lem:shiftdiv}
For all $t > 0$ and $n > 0$ and $l \geq 0$, we have an isomorphism $\nSh_t(\Div_n\{\bV\}^{(l)}) \cong \Div_n\{\bV\}^{(l)} \oplus W^{(l)}$ of $\GL$-representations where $W$ is a direct sum of finite copies of $\Div_i\{\bV\}$ with $0 \leq i < n$.
\end{lemma}
\begin{proof}
The functor $\nSh_t$ commutes with taking Frobenius twists, so it suffices to prove the result for $l = 0$. 
It is clear that the functor $\nSh_t$ maps $\bV$ to $\bV \oplus k^t$ so $\nSh_t(\Div_n\{\bV\}) = \Div_n\{\nSh_t(\bV)\} = \Div_n\{k^t \oplus \bV\} = \bigoplus_{i = 0}^n \Div_i\{k^t\} \otimes \Div_{n-i}\{\bV\}$, as required.
\end{proof}

\begin{lemma}\label{lem:premet}
    Let $\lambda$ be a nonempty $p$-restricted partition and $q = p^r$. We have $\mShq(L_{\lambda}^{(r)}) \ne 0$.
\end{lemma}
\begin{proof}
   By \cite[Proposition~4.4]{gan25pol}, there exists a weight vector of weight $(q^{|\lambda|})$ in $L_{\lambda}^{(r)}$, so $\mShq(L_{\lambda}^{(r)}) \ne 0$.
\end{proof}

We can now prove the key technical result needed for the shift theorem for $\Drs$.
\begin{lemma}\label{lem:newlemma}
    Assume $W$ is a polynomial representation of $\GL$. Assume $n$ is a positive integer divisible by $p^r$. Any nonzero subrepresentation $U \subset (\Drs)_n \otimes W$ satisfies $\mShq(U) \ne 0$.
\end{lemma}
\begin{proof}
Since $\Drs \subset \Drinfty$, it suffices to prove the result with $s$ replaced with $\infty$. By \cite[Proposition~4.8]{gan25pol}, we may choose a \textit{tensor-disjoint} weight vector $u \in U$, i.e., we can write $u = \sum d_i \otimes w_i$ with the support of the weights of $d_i$ and $w_i$ are disjoint for all $i$.
Applying a permutation matrix if necessary, we may further assume that the $t$-th component of $\wgt(d_i)$ is nonzero for some $i$. Write $\nSh_{t-1}((\Drinfty)_n) =  X_1 \oplus X_0$ where $X_1$ is strictly of positive degree and $X_0$ is of degree $0$; similarly, write $\nSh_{t-1}(W) = Y_1 \oplus Y_0$. 
Let $U'$ be the $G(t-1)$-subrepresentation generated by $u$. 
We have $U' \subset (X_0 \otimes Y_1) \oplus (X_1 \otimes Y_0)$ as it is tensor-disjoint. Since we also assumed $\wgt(d_i)$ is nonzero in the $t$-th component, the image of $U'$ in $X_1 \otimes Y_0$ is nonzero. Using Lemma~\ref{lem:shiftdiv}, we see that the $\GL$-representation $X_1 \otimes Y_0$ is a direct sum of $(\Div_i\{\bV\})^{(r)}$ with $0 < i \leq \frac{n}{p^r} $. Thus $U'$ contains $L_{p^r \nu_i}$ for some $0 < i \leq \frac{n}{p^r}$ by Lemma~\ref{lem:socle}. By Lemma~\ref{lem:premet}, we have $\mShq(L_{p^r \nu_i}) \ne 0$ so by exactness of $\mShq$, we get $\mShq(U') \ne 0$ but since $U' \subset U$, we also have $\mShq(U) \ne 0$, as required.
\end{proof}

\subsection{Vanishing of \texorpdfstring{$\mDeq$}{Delta}} \label{ss:deltavanishing}
 In this section, we prove a partial converse to Corollary~\ref{cor:dfrobzero}: we show that a finitely generated torsion-free $\Drs$-module $M$ with $\mDeq(M) = 0$ is semi-induced.  The  proof here is entirely analogous to \cite[Section~4.3]{gan22ext} and \cite[Section~5.3]{gan25pol} given the results of Section~\ref{ss:technical}.

\begin{lemma}\label{lem:handlingtorsion}
Let $0 \to L \to M \to N \to 0$ be a short exact sequence of $\Drs$-modules. We have a six-term exact sequence
   \begin{displaymath}
       0 \to \bKq(L) \to \bKq(M) \to \bKq(N) \to \mDeq(L) \to \mDeq(M) \to \mDeq(N) \to 0.
   \end{displaymath}
\end{lemma}
\begin{proof}
    The result follows by applying the snake lemma to the diagram
\begin{displaymath}
\begin{tikzcd}
  0 \arrow[r] & L \arrow[d, "i_L"] \arrow[r] & M \arrow[d, "i_M"] \arrow[r] & N\arrow[d, "i_N"] \arrow[r] & 0 \\
  0 \arrow[r] & \mShq(L) \arrow[r] & \mShq(M) \arrow[r] & \mShq(N) \ar[r] & 0.
\end{tikzcd}
\end{displaymath}
\end{proof}

\begin{corollary}\label{cor:handlingtorsion}
Let $0 \to L \to M \to N \to 0$ be a short exact sequence of $\Drs$-modules such that $N$ is torsion-free. We have a short exact sequence 
\begin{displaymath}
0 \to \mDeq(L) \to \mDeq(M) \to \mDeq(N) \to 0
\end{displaymath}
\end{corollary}
\begin{proof}
The module $\bKq(N) = 0$ by Proposition~\ref{prop:basicshift}(a), so the result follows from the previous lemma.
\end{proof}

\begin{lemma}\label{lem:dmzero}
    Let $M$ be a finitely generated torsion-free $\Drs$-module with $M_i = 0$ for $i < n$ and $\mShq(M_n) \ne 0$. Then $\mDeq(M) \ne 0$. 
\end{lemma}
\begin{proof}
    The image of the map $M \to \mShq(M)$ is in degree $\geq n$ so $\mDeq(M) \ne 0$ as $\mShq(M)$ is nonzero in degree $n-q$ by assumption.
\end{proof}

\begin{proposition}\label{prop:dmzerogenonedeg}
    Let $M$ be a finitely generated torsion-free $\Drs$-module generated in degree $n$ such that $\mShq(M_n) = 0$. The natural map $\Drs \otimes M_n \to M$ is an isomorphism, i.e., $M$ is an induced $\Drs$-module.
\end{proposition}
\begin{proof}
    Consider the natural map $\phi \colon \Drs \otimes M_n \to M$. The map $\phi$ is surjective, and $K = \ker(\phi)$ is zero in degrees $\le n$. We have to prove that $K = 0$. First, we claim that $\mDeq(K) = 0$. Indeed, since $M$ is torsion-free, by Corollary~\ref{cor:handlingtorsion}, we have a short exact sequence 
     \begin{displaymath}
    0 \to \mDeq(K) \to \mDeq(\Drs \otimes M_n) \to \mDeq(M) \to 0.
    \end{displaymath} 
    By Corollary~\ref{cor:dfrobzero}, we have $\mDeq(\Drs \otimes M_n) = 0$ from the assumption that $\mShq(M_n)= 0$ and so $\mDeq(K) = 0$ from the above short exact sequence, as claimed. 
    Now, suppose $K$ is nonzero, and let $m$ be the smallest degree such that $K_m \ne 0$. Since $m > n$, the $\GL$-representation $K_m$ is a subrepresentation of $({\Drs})_{+} \otimes M_n$. 
    By Lemma~\ref{lem:newlemma}, we have $\mShq(K_m) \ne 0$, so 
    applying Lemma~\ref{lem:dmzero}, we see that $\mDeq(K) \ne 0$, which is a contradiction. Therefore, the map $\phi$ is an isomorphism, as required.
\end{proof}

\begin{lemma}\label{lem:torsionfreequotient}
    Let $M$ be a finitely generated torsion-free $\Drs$-module with $t_0(M) = n$ such that $\mDeq(M)$ is semi-induced and $\mDeq(M/M^{<n}) = 0$. Then $M/M^{<n}$ is torsion-free.
\end{lemma}
\begin{proof}
    The exact sequence from Lemma~\ref{lem:handlingtorsion} becomes
    \begin{displaymath}
    0 \to \mKq(M/M^{<n}) \to \mDeq(M^{<n}) \to \mDeq(M) \to 0.
    \end{displaymath} 
    The associated long exact sequence of $\Tor^A(-,k)$ gives us the exact sequence
    \begin{displaymath}
    \Tor_1^{\Drs}(\mDeq(M), k) \to \Tor_0^{\Drs}(\mKq(M/M^{<n}),k) \to \Tor_0^{\Drs}(\mDeq(M^{<n}),k).
    \end{displaymath}
    However, since $\mDeq(M)$ is semi-induced, by Proposition~\ref{prop:flatequalssemi}, we get $\Tor_1^{\Drs}(\mDeq(M), k) = 0$, which implies from the above exact sequence that 
    \begin{displaymath}
    t_0(\mKq(M/M^{<n})) \leq t_0(\mDeq(M^{<n})) < n.
    \end{displaymath}
    But $\mKq(M/M^{<n})$ is supported only in degrees $\ge n$ as it is a submodule of $M/M^{<n}$, so must be zero. Therefore, the natural map $M/M^{<n} \to \mShq(M/M^{<n})$ is injective, which implies that the module $M/M^{<n}$ is torsion-free by Proposition~\ref{prop:basicshift}(b).
\end{proof}
We can now prove the result promised at the beginning of this section.
\begin{proposition}\label{prop:dmzeroisflat}
Let $M$ be a finitely generated torsion-free $\Drs$-module with $\mDeq(M) = 0$. The $\Drs$-module $M$ is semi-induced.
\end{proposition}
\begin{proof}
    We proceed by induction on the generation degree of $M$.
    When $t_0(M) = 0$, this follows from Proposition~\ref{prop:dmzerogenonedeg}.
    Assume that $t_0(M) = n > 0$. Since $\mDeq$ is right exact, we see that $\mDeq(M/M^{<n}) = 0$. Therefore, by Lemma~\ref{lem:torsionfreequotient}, the module $M/M^{<n}$ is torsion-free, and so by Proposition~\ref{prop:dmzerogenonedeg}, the module $M/M^{<n}$ is semi-induced. By Corollary~\ref{cor:handlingtorsion}, we have a short exact sequence 
    \begin{displaymath}
    0 \to \mDeq(M^{<n}) \to \mDeq(M) \to \mDeq(M/M^{<n}) \to 0.
    \end{displaymath}
    As $\mDeq(M)=0$, we have that $\mDeq(M^{<n}) = 0$, which implies that $M^{<n}$ is semi-induced (by induction) and therefore, so is $M$. 
\end{proof}
We also note a corollary of Lemma~\ref{lem:torsionfreequotient} that we will use in the proof of the shift theorem.
\begin{corollary}
 \label{cor:dqlowdeg}
     Let $M$ be a finitely generated torsion-free $\Drs$-module with $t_0(M) = n$ such that $\mDeq(M)$ is semi-induced with $t_0(\mDeq(M)) \leq \max(-1, n - q - 1)$. Then $M/M^{<n}$ is semi-induced. 
 \end{corollary}
 \begin{proof}
     Since the functor $\mDeq$ is right exact, we have \begin{displaymath}
     t_0(\mDeq(M/M^{<n})) \le t_0(\mDeq(M)) < \max(n-{q},0)
     \end{displaymath}
     Since $M/M^{<n}$ is supported in degrees $\ge n$, the module $\mShq(M/M^{<n})$ is supported in degrees $\ge n-{q}$, and so $\mDeq(M/M^{<n})$ is also supported in degrees $\geq n-{q}$. Therefore, the above inequality implies that $\mDeq(M/M^{<n}) = 0$. Since $\mDeq(M)$ is semi-induced by assumption, Lemma~\ref{lem:torsionfreequotient} now applies and we get that the module $M/M^{<n}$ is torsion-free. Since $\mDeq(M/M^{<n}) = 0$, by Lemma~\ref{lem:dmzero}, we see that $\mShq((M/M^{<n})_n) = 0$. The assumptions of Proposition~\ref{prop:dmzerogenonedeg} are now satisfied for $M/M^{<n}$ so the $\Drs$-module $M/M^{<n}$ is semi-induced, as required.
 \end{proof}

\subsection{Proof of the shift theorem}\label{ss:proofshift}

\begin{lemma}
	Let $M$ be a $\Drs$-module. For all $n$, we have $\maxdeg(M\{k^n\}) \leq (p^{s+1}-p^r)n + t_0(M)$. Furthermore, if $M$ is semi-induced, then $\maxdeg(M\{k^n\}) = {(p^{s+1}-p^r)n + t_0(M)}$ for all sufficiently large $n$.
\end{lemma}
\begin{proof}
    Set $\alpha = p^{s+1} - p^r$. The lemma follows from the observation that $\Drs$ evaluated at $k^n$ is supported in degrees $ \leq \alpha n$ and nonzero in that degree since the monomial $x_1^{[\alpha]} x_2^{[\alpha]} \ldots x_n^{[\alpha]} \ne 0$ in $\Drs\{k^n\}$. 
\end{proof}

\begin{corollary}\label{cor:subsemi}
Let $F$ be a finitely generated semi-induced $\Drs$-module and $Z$ be a submodule of $F$ such that $Z/Z^{< t_0(Z)}$ is semi-induced. Then $t_0(Z) \leq t_0(F)$.
\end{corollary}
\begin{proof}
Let $\alpha = p^{s+1} - p^r$. For sufficiently large $n$, $\maxdeg(Z\{k^n\}) = \alpha n + t_0(Z)$ since $Z/Z^{< t_0(Z)}\{k^n\}$ is nonzero in that degree by the previous lemma. Since $Z$ is a submodule of $F$, we have $\maxdeg(Z\{k^n\}) \le \maxdeg(F\{k^n\})$ for sufficiently large $n$, or $\alpha n + t_0(Z) \leq \alpha n + t_0(F)$ for large $n$, giving us the required inequality.
\end{proof} 


The next lemma distills the technical results we have proved so far in this section into what is needed in the proof of the shift theorem.
\begin{lemma}\label{lem:t1lesst0}
    Let $M$ be a finitely generated torsion-free $\Drs$-module such that $\mDeq(M)$ is semi-induced. Then $t_1(M) \leq t_0(M)$.
\end{lemma}

\begin{proof}
     Let $F$ be a semi-induced module surjecting onto $M$ such that $t_0(F) = t_0(M)$. We have an exact sequence 
     \begin{displaymath}
     0 \to Z \to F \to M \to 0. 
     \end{displaymath}
     The long exact sequence of $\Tor^{\Drs}(-, k)$ yields:
    \begin{displaymath}
    0 \to \Tor_1^{\Drs}(M, k) \to \Tor_0^{\Drs}(Z, k) \to \Tor_0^{\Drs}(F, k) \to \Tor_0^{\Drs}(M, k) \to 0. 
    \end{displaymath}
    Therefore, we have the inequality $t_1(M) \le t_0(Z)$. So it suffices to show that $t_0(Z) \le t_0(M)$.
    
    Since $M$ is torsion-free, we also obtain a short exact sequence 
    \begin{displaymath}
    0 \to \mDeq(Z) \to \mDeq(F) \to \mDeq(M) \to 0
    \end{displaymath}
    by Corollary~\ref{cor:handlingtorsion}. The module $\mDeq(F)$ is semi-induced by Proposition~\ref{prop:basicshift}(e), and $\mDeq(M)$ is semi-induced by assumption. So by Corollary~\ref{cor:semises}, the module $\mDeq(Z)$ is also semi-induced.
    
    We know $t_0(\mDeq(Z)) \le t_0(Z) - q$ by Proposition~\ref{prop:basicshift}(d). If $t_0(\mDeq(Z)) < t_0(Z) - q$, then $Z/Z^{< t_0(Z)}$ is semi-induced by Corollary~\ref{cor:dqlowdeg} (note that $Z$ is torsion-free being a submodule of $F$) and therefore, by Corollary~\ref{cor:subsemi}, we have $t_0(Z) \leq t_0(F) = t_0(M)$. 
    
    Instead, if $t_0(\mDeq Z) = t_0(Z) - q$, then from the short exact sequence above, we have the inequality 
    \begin{displaymath}
    t_0(\mDeq(Z)) = t_0(Z) - q \le t_0(\mDeq(F)) \le t_0(F) - q = t_0(M) - q,
    \end{displaymath}
    or $t_0(Z) \le t_0(M)$. 
\end{proof}

\begin{proposition}\label{prop:liftingsemi}
	Let $M$ be a finitely generated torsion-free $\Drs$-module such that $\mDeq(M)$ is semi-induced. Then $M$ is also semi-induced.
\end{proposition}
\begin{proof}
	We induct on $t_0(M)$. When $t_0(M) = 0$, the module $\mDeq(M) = 0$, and so $M$ is semi-induced by Proposition~\ref{prop:dmzeroisflat}. Assume now $t_0(M) = n > 0$. 
	We have the short exact sequence
	\begin{displaymath}
	0 \to M^{<n} \to M \to M/M^{<n} \to 0.
	\end{displaymath}
	So it suffices to prove that $M^{<n}$ and $M/M^{<n}$ are semi-induced.
	
	We first show that $M/M^{<n}$ is semi-induced. The long exact sequence of $\Tor^{\Drs}(-, k)$ associated to the above short exact sequence gives us 
	\begin{displaymath}
    \Tor_1^{\Drs}(M, k) \to \Tor_1^{\Drs}(M/M^{<n}, k) \to \Tor_0^{\Drs}(M^{<n}, k)
    \end{displaymath}
	from which we see that $t_1(M/M^{<n}) \leq \max(t_1(M), t_0(M^{<n})) = \max(t_1(M), n-1) \leq n$; for the last inequality, we use Lemma~\ref{lem:t1lesst0} to get $t_1(M) \leq t_0(M) = n$. By Lemma~\ref{lem:relationsinlowdeg}, we see that $M/M^{<n}$ is a semi-induced module. 
	
    We now proceed to show that $M^{<n}$ is also semi-induced. Since $M/M^{<n}$ is semi-induced, it is torsion-free, and therefore, we have a short exact sequence
    \begin{displaymath}
    0 \to \mDeq(M^{<n}) \to \mDeq(M) \to \mDeq(M/M^{<n}) \to 0
    \end{displaymath}
    by Corollary~\ref{cor:handlingtorsion}. By Proposition~\ref{prop:basicshift}(e), we have that $\mDeq(M/M^{<n})$ is semi-induced, and so from the short exact sequence above, we see that $\mDeq(M^{<n})$ is also semi-induced by Corollary~\ref{cor:semises}. Since $t_0(M^{<n}) < n$, by the induction hypothesis, it now follows that $M^{<n}$ is semi-induced, as required. 
\end{proof}
We can now prove the shift theorem from the introduction using a standard induction technique of Li--Yu \cite{ly17fi}.
\begin{proof}[Proof of Theorem~\ref{thm:shiftintro}]
We follow the proof of \cite[Theorem~3.13]{ly17fi} and proceed by induction on $t_0(M)$. It suffices to prove the result for $\mShq^t(M)$ for $t \gg 0$, and so we may additionally assume $M$ is torsion-free by Proposition~\ref{prop:basicshift}(i). When $t_0(M) = 0$, the module $\mDeq(M)=0$ and so $M$ is semi-induced by Proposition~\ref{prop:liftingsemi}. Assume now $t_0(M) = n > 0$. We have a short exact sequence $0 \to M \to \mShq(M) \to \mDeq(M) \to 0$ with $t_0(\mDeq(M)) < n$. By induction, for $l\gg 0$ the module $\mShq^{l}(\mDeq(M))$ is semi-induced. By Proposition~\ref{prop:basicshift}(g) we have $\mShq^{l}(\mDeq(M)) \cong \mDeq(\mShq^{l}(M))$, and so by Proposition~\ref{prop:liftingsemi} we have that $\mShq^l(M)$ is semi-induced, as required. 
\end{proof}

\subsection{Generators of the derived category}\label{ss:gensdbDrs}
We prove two useful results for $\Mod_{\Drs}$, which essentially state that a finitely generated module has a finite right resolution by finitely generated flat modules up to possibly torsion homology.

\begin{proposition}[Resolution Theorem]\label{prop:resolutionDrs}
Let $M$ be a finitely generated $\Drs$-module. We have a chain complex of $\Drs$-modules
\begin{displaymath}
0 \to M \to P^0 \to P^1 \to \ldots \to P^m \to 0
\end{displaymath}
satisfying the following properties:
\begin{itemize}
\item each $P^i$ is a finitely generated semi-induced module with $t_0(P^i) \le t_0(M) - qi$, and
\item the cohomology of this complex is a torsion $\Drs$-module.
\end{itemize}
{Furthermore, given a map of $\Drs$-modules $f \colon M \to N$, we can choose complexes $M \to P^{\bullet}$ and $N \to Q^{\bullet}$ satisfying the above properties, and a map of complexes $\tilde{f}\colon P^{\bullet} \to Q^{\bullet}$ extending $f$.}
\end{proposition}
\begin{proof}
The proof of \cite[Theorem~5.1]{gan22ext} applies. 
\end{proof}

For a nonzero ideal $\fa \subset A$, we let $\Mod_A[\fa^{\infty}]$ be the subcategory of $\Mod_A$ containing modules that are locally annihilated by $\fa$. 
\begin{proposition}[Property (Inj)]\label{prop:propinj}
    Assume $s < \infty$ and that $I$ is an injective object in $\Mod_{\Drs}[\Jrd^{\infty}]$. The $\Drs$-module $I$ is also injective in $\Mod_{\Drs}$.
\end{proposition}
\begin{proof}
The Rees algebra $\Drs[\Jrd t]$ is $\fS_{\infty}$-noetherian as it is a quotient of $\Sym(\bV^{s-r+1}) \otimes \Sym(\bV^{d-r+1})$ which is $\fS_{\infty}$-noetherian by Cohen's theorem \cite{co67laws}. The claimed result now follows by the Artin--Rees lemma; see \cite[Section~4.4]{ss19gl2} for details.
\end{proof}

\begin{proposition}\label{prop:triangles}
    Let $M$ be a finitely generated $\Drs$-module. We have a triangle $$T \to M \to F \to$$ in $\rD^b_{\fgen}(\Mod_{\Drs})$ where $T$ is quasi-isomorphic to a bounded complex of finitely generated torsion $\Drs$-modules, and $F$ is quasi-isomorphic to a bounded complex of semi-induced $\Drs$-modules. 
\end{proposition}
\begin{proof}  Let $M \xrightarrow{f} P^{\bullet}$ be a complex satisfying the conclusion of Proposition~\ref{prop:resolutionDrs}. The cone of $f$ is a bounded complex with finitely generated torsion cohomology, so it is quasi-isomorphic to a complex of finitely generated torsion $\Drs$-modules by Proposition~\ref{prop:propinj} and \cite[Lemma~6.5]{gan25pol}. Therefore, the distinguished triangle $\cone(f)[-1] \to M \to P^{\bullet} \to$ satisfies the requisite properties. 
\end{proof}

\begin{lemma}\label{lem:torsionDrr}
    A finitely generated torsion $D_{[r,r]}$-module is finite length.
\end{lemma}
\begin{proof}
   By Lemma~\ref{lem:torsiondevissage}, a torsion $D_{[r,r]}$-module has a finite filtration by $D_{[r,r]}/\Jrr \cong k$, so is finite length.
\end{proof}

\begin{theorem}\label{thm:gensdbDrs}
The bounded derived category $\rD^b_{\fgen}(\Mod_{\Drs})$ is generated by the classes of modules $L_{\lambda}$ and $D_{[e,s]} \otimes L_{\lambda}$ with $\lambda$ varying over all partitions and $e$ varying over all integers in $[r, s]$.
\end{theorem}
\begin{proof}
We proceed by induction on $(s-r)$. When $(s - r) = 0$, this follows by Proposition~\ref{prop:triangles} and Lemma~\ref{lem:torsionDrr}. Now assume $(s - r) > 0$.  Another application of Proposition~\ref{prop:triangles} shows that every finitely generated module, and hence every object in $\rD^b_{\fgen}(\Mod_{\Drs})$ is in the triangulated subcategory generated by modules of the form $\Drs \otimes L_{\lambda}$ and torsion $\Drs$-modules. 
By Lemma~\ref{lem:torsiondevissage}, a finitely generated torsion $\Drs$-module is filtered by modules over $D_{[r+1, s]}$; such modules are in the triangulated subcategory generated by finite length modules and modules of the form $D_{[e, s]} \otimes L_{\lambda}$ with $r+1 \leq e \leq s$ by the induction hypothesis. The claim follows by combining the previous two sentences.
\end{proof}

\subsection{Local cohomology}\label{ss:lcprelim}
We fix an integer $d$ satisfying $r \leq d \leq s < \infty$, and set $u = p^d$ and $q = p^r$. Let $\mGau$ be the functor which sends a $\Drs$-module to the submodule of elements locally annihilated by $\Jrd$. The functor $\mGau$ is right adjoint to the inclusion functor $\Mod_{\Drs}[\Jrd^{\infty}] \to \Mod_{\Drs}$. The \textit{local cohomology modules} supported at $\Jrd$ are the right derived functors of $\mGau$. 

\begin{lemma}\label{lem:lccolim}
    Let $M$ be a $\Drs$-module. Assume $M = \colim M_{\alpha}$ is a filtered colimit in $\Mod_{\Drs}$. Then $\rR^i\mGau(M) = \colim \rR^i\mGau(M_{\alpha})$ for all $i$.
\end{lemma}
\begin{proof}
The functor $\mGau$ commutes with colimits, so the result follows by \cite[Proposition~A.4]{gs18incmon}.
\end{proof}

\begin{proposition}\label{prop:torsionlc}
    Assume $M$ is a $\Drs$-module locally annihilated by $\Jrd$. The local cohomology modules $\rR^i\mGau(M) = 0$ for all $i > 0$ and $\mGau(M) = M$.
\end{proposition}
\begin{proof}
This is a consequence of Proposition~\ref{prop:propinj}; see \cite[Section~4]{ss19gl2} for more details. 
\end{proof}

\subsubsection*{Local cohomology modules supported at the height one prime}
Of particular importance for $\Drs$-modules is the functor $\mGaq$: by Lemma~\ref{lem:torsionJrr}, the functor $\mGaq$ is right adjoint to the inclusion functor $\Mod_{\Drs}^{\tors} \to \Mod_{\Drs}$.  So we will focus on the local cohomology modules supported at the height one prime ideal $\Jrr \subset \Drs$. 

We continue with our assumption $s < \infty$ and set $q = p^r$. In Proposition~\ref{prop:basicshift}(h), we saw that the functors $\mGaq$ and $\mShq$ commute. Our main technical result is that $\mShq$ commutes with the derived functors of $\mGaq$. 

Let $i(l) \colon \id \to \mShq^l$ be the map with $m \mapsto y_l^{[q]}y_{l-1}^{[q]} \ldots y_2^{[q]}y_1^{[q]}m$. We let $\mKql = \ker(i(l))$ and $\mDel = \coker(i(l))$.

\begin{lemma}\label{lem:mDelcommmShq}
The functor $\mDel$ commutes with $\mShq$.
\end{lemma}
\begin{proof}
Note that $i(l) = i_{\mShq^{l-1}(M)} \circ i_{\mShq^{l-2}(M)} \circ \ldots \circ i_M$. The result follows by a repeated application of Lemma~\ref{lem:obvshiftlemma}.
\end{proof}

\begin{lemma}
    For an arbitrary $\Drs$-module $M$, we have $\mKql(M) = M[\Jrr^l]$.
\end{lemma}
\begin{proof}
  Let $M$ be an arbitrary $\Drs$-module. An element $m \in \ker(i(l)_M)$ if and only if 
   $y_l^{[q]}y_{l-1}^{[q]} \ldots y_1^{[q]}m = 0$. This equality remains true in the $GL(k^l \oplus \bV)$-equivariant $\Drs\{k^l \oplus \bV\}$-module $M\{k^l \oplus \bV\}$. Since $m \in M\{\bV\}$, the equality  $y_l^{[q]}y_{l-1}^{[q]} \ldots y_1^{[q]}m = 0$ is satisfied if and only if $\langle y_l^{[q]}y_{l-1}^{[q]} \ldots y_1^{[q]}\rangle \langle m \rangle = 0$ by Lemma~\ref{lem:disjointnzd}. Another application of the same lemma yields that $\langle y_l^{[q]}y_{l-1}^{[q]} \ldots y_1^{[q]} \rangle =  \langle y_l^{[q]} \rangle \langle y_{l-1}^{[q]} \rangle  \ldots \langle y_1^{[q]} \rangle = \langle x_1^{[q]} \rangle^l$.  Therefore, the element $m \in \mKql(M)$ if and only if $\Jrr^l$ annihilates $m$, as required.
\end{proof}

\begin{corollary}\label{cor:colim}
For a torsion $\Drs$-module $M$, we have $M = \colim_l \mKql(M)$.
\end{corollary}
\begin{proof}
   By the previous lemma, we have $\mKql(M) = M[\Jrr^l]$, and by Lemma~\ref{lem:torsionJrr}, we have  $M = M[\Jrr^{\infty}]$, whence the result follows.
\end{proof}

We now explain the set up leading up to Proposition~\ref{prop:shcommutesderivedgamma}. Let $M$ be an $\Drs$-module, and let $M \to I^\bullet$ and $\mShq(M) \to J^\bullet$ be injective resolution of $M$ and $\mShq(M)$ respectively. The identity map of $\mShq(M)$ induces a map of complexes $\mShq(I^\bullet) \to J^\bullet$ by the lifting property of the injective resolution $J^\bullet$. Applying $\mGaq$ and taking cohomology, we obtain a map $ H^i(\mGaq(\mShq(I^\bullet))) \to H^i(\mGaq(J^\bullet)) = \rR^i\mGaq(\mShq(M))$. We also have natural isomorphisms,
\begin{displaymath}H^i(\mGaq(\mShq(I^\bullet))) \cong H^i(\mShq(\mGaq(I^\bullet))) \cong \mShq(H^i(\mGaq(I^\bullet))) = \mShq(\rR^i\mGaq(M)),
\end{displaymath}
where for the first isomorphism, we use that $\mShq$ commutes with $\mGaq$, and for the second isomorphism, we use that $\mShq$ commutes with taking cohomology as it is an exact functor.
Putting all this together, we get a natural map $F_{i,M} \colon \mShq(\rR^i\mGaq(M)) \to \rR^i\mGaq(\mShq(M))$.
Our proof of this next result is an expanded version of the proof of \cite[Proposition~A.3]{dj16lc}.
\begin{proposition}\label{prop:shcommutesderivedgamma}
The map $F_i$ defined above is an isomorphism for all $i \ge 0$.
\end{proposition}
\begin{proof}
  We proceed by induction on $i$ with the $i=0$ case being Proposition~\ref{prop:basicshift}(h). Now, assume $i > 0$, and we know that $F_{i-1}$ is an isomorphism. We have to show that $F_i$ is an isomorphism. We make two reductions. 
 
 First, we may assume that $M$ is torsion-free. Indeed, given an arbitrary module $N$, we have the short exact sequence $0 \to \mGaq(N) \to N \to N/\mGaq(N) \to 0$. Using the long exact sequence of local cohomology and Proposition~\ref{prop:torsionlc}, we see that for $i > 0$, the map $F_{i, N}$ is an isomorphism if and only if $F_{i, N/\mGaq(N)}$ is an isomorphism. 
 
 Second, we may assume that $M$ is finitely generated. Indeed, for an arbitrary module $N$, the map $F_{i, N}$ is the colimit of the maps $\{F_{i, N_{\alpha}}\}_{\alpha}$ where $\{N_\alpha\}_{\alpha}$ is the directed system of finitely generated submodules of $N$. By Lemma~\ref{lem:lccolim}, we see that it suffices to show  $F_{i, M}$ is an isomorphism for finitely generated $M$. So for the remainder of the proof, we assume $M$ is finitely generated and torsion-free.
 
 We now show that $F_{i,M}$ is injective. Embedding $M$ into an injective $\Drs$-module, we obtain a short exact sequence $ 0 \to M \to I \to N \to 0.$ 
 The associated long exact sequence gives a commutative diagram
 \[
  \begin{tikzcd}[sep=small]
     \mShq(\rR^{i-1}\mGaq(I)) \arrow[d] \arrow[r] & \mShq(\rR^{i-1}\mGaq(N)) \arrow[d] \arrow[r] & \mShq(\rR^i \mGaq(M)) \arrow[d] \arrow[r] & \mShq(\rR^i\mGaq(I)) = 0 \arrow[d] \\
    \rR^{i-1}\mGaq(\mShq(I)) \arrow[r] &  \rR^{i-1}\mGaq(\mShq(N)) \arrow[r] & \rR^{i}\mGaq(\mShq(M)) \arrow[r] & \rR^i\mGaq(\mShq(I)) 
  \end{tikzcd}
 \]
with exact rows. The first two vertical maps are isomorphisms by induction and the fourth vertical map is injective. Thus by the four lemma, the third vertical map is injective as claimed. By induction on $l$, we also get that the maps $\mShq^l\rR^i\mGaq(M) \to \rR^i\mGaq(\mShq^l(M))$ are injective for all $l \in \bN$.

We now set $\kappa_l(M) = \ker(i(l)_{\rR^i \mGaq(M)})$, i.e., set $\kappa_l(M) = \mKql(\rR^i\mGaq(M))$. We will now show that the modules $\mShq\kappa_l(M)$ and $\kappa_l(\mShq(M))$ are isomorphic for all $l$. Since $M$ is torsion-free, we have a short exact sequence $0 \to M \to \mShq^l(M) \to \mDel(M) \to 0,$
  which yields the exact sequence
 \begin{displaymath}
 \rR^{i-1}\mGaq(\mShq^l(M)) \to \rR^{i-1}\mGaq(\mDel(M)) \to \rR^i \mGaq(M) \to \rR^i \mGaq(\mShq^l(M)).
  \end{displaymath}
 {Let $I^{\bullet}$ and $J^{\bullet}$ be injective resolutions of $M$ and $\mShq^l(M)$ respectively}. We can choose a lift of the map $i(l)$ to the injective resolution which factors $I^{\bullet} \to \mShq^l(I^{\bullet}) \to J^{\bullet}$. Therefore, the last map in the previous exact sequence factors as
 $\rR^i \mGaq(M) \to \mShq^l \rR^i \mGaq(M) \to \rR^i \mGaq(\mShq^l(M)).$ 
 We have already shown that $\mShq^l \rR^i \mGaq(M) \to \rR^i\mGaq(\mShq^l(M))$ is injective, so we get that $\kappa_l(M)$ can also be written as $\ker[\rR^i \mGaq(M) \to \rR^i\mGaq (\mShq^l(M))]$ and in turn, we have
 \begin{displaymath}
 \kappa_l(M) \cong \coker[\rR^{i-1} \mGaq (\mShq^l(M)) \to \rR^{i-1} \mGaq(\mDel (M))].
 \end{displaymath}
We thus obtain the following chain of isomorphisms proving our claim that $\mShq(\kappa_l(M)) \cong \kappa_l(\mShq(M))$ for all $l$:
 \begin{align*}
\mShq \kappa_l(M) & = \mShq \ker[\rR^i\mGaq(M) \to \rR^i\mGaq(\mShq^l(M))] \\
&\cong \mShq \coker[\rR^{i-1} \mGaq (\mShq^l(M)) \to \rR^{i-1} \mGaq(\mDel (M))] \\
&\cong \coker[\mShq \rR^{i-1} \mGaq (\mShq^l(M)) \to \mShq\rR^{i-1}\mGaq(\mDel(M))] \\
&\cong \coker[\rR^{i-1} \mGaq (\mShq(\mShq^l(M))) \to \rR^{i-1}\mGaq(\mShq(\mDel(M))) ]\\
&\cong \coker[\rR^{i-1} \mGaq (\mShq^l(\mShq(M))) \to \rR^{i-1}\mGaq(\mDel(\mShq(M))) ]\\
&\cong \ker[\rR^i\mGaq(\mShq(M)) \to \rR^i\mGaq(\mShq^l(\mShq(M)))] \\
&= \kappa_l(\mShq(M)).
 \end{align*}
  where
  the first equality uses the definition of $\kappa_l$ and the fact that $F_i$ is injective, then we use the isomorphism from the preceding paragraph
  to get to the second line from the first, that $\mShq$ commutes with taking cokernels to get to the third line, the induction hypothesis for the fourth line, that $\mShq$ commutes with $\mShq^l$ and $\mDel$ for the fifth, and the reasoning for the sixth and seventh line are the same as for the second and first line respectively.
 
 Now, since $\rR^i \mGaq(M)$ is a torsion module, we have $\rR^i\mGaq(M) = \colim_{\substack{l \in \bN}} \kappa_l(M)$ by Corollary~\ref{cor:colim}, from which we get another chain of isomorphisms:
 \begin{align*}
  \mShq \rR^i \mGaq(M) &= \mShq (\colim \kappa_l(M)) \\ 
  & \cong \colim \mShq(\kappa_l(M)) \\
  & \cong \colim \kappa_l(\mShq(M)) \\
  & = \rR^i \mGaq(\mShq(M)),
  \end{align*}
  where to get to the second line, we use the fact that $\mShq$ commutes with colimits, and to get to the third line, we use the fact that the functor $\mShq$ commutes with $\kappa_l$ which we proved in the previous paragraph. Therefore, the natural map $F_{i, M}$ is an isomorphism, as required.
\end{proof}

\begin{corollary}\label{cor:lcvanishing}
  Assume $M$ is a finitely generated semi-induced $\Drs$-module. For all $i\in \bN$, we have $\rR^i \mGaq (M) = 0$.
\end{corollary}
\begin{proof}
By d\'evissage, we are reduced to proving the result for induced $\Drs$-modules, so we assume $M$ is induced.
It suffices to show that the natural map $\rR^i \mGaq(M) \to \mShq(\rR^i\mGaq(M))$ is injective, as Proposition~\ref{prop:basicshift} will then imply that $\rR^i\mGaq(M)$ is torsion-free but the local cohomology modules are torsion modules, forcing it to be zero. The map $M \to \mShq(M)$ is split, hence the map $\rR^i\mGaq(M) \to \rR^i\mGaq(\mShq(M))$ is also split and thus injective. The injectivity of the requisite map now follows by Proposition~\ref{prop:shcommutesderivedgamma}.
\end{proof}

\begin{corollary}\label{cor:semiinj}
   A semi-induced $\Drs$-module has an injective resolution by torsion-free $\Drs$-injectives. 
\end{corollary}
\begin{proof}
    Let $T \colon \Mod_{\Drs} \to \Mod_{\Drs}/\Mod_{\Drs}^{\tors}$ be the exact quotient functor and $S$ be its right adjoint. The functor $T$ preserves injectives by \cite[Proposition~4.3]{ss19gl2} and $S$ also preserves injectives being right adjoint to an exact functor. 
    Let $M \to I^{\bullet}$ be an injective resolution of the semi-induced $\Drs$-module $M$. Applying the exact functor $T$, we see that $T(M) \to T(I^{\bullet})$ is an injective resolution of $T(M)$ in $\Mod_{\Drs}/\Mod_{\Drs}^{\tors}$. Applying $S$ to this complex again, and using the fact that $ST(M) \cong M$, we get an injective resolution of $M$ and $ST(I)$ is evidently torsion-free.
\end{proof}

For $\Drs$-modules, we can now easily deduce finiteness results for local cohomology supported at the height one prime $\Jrr$ using Theorem~\ref{thm:gensdbDrs}.
\begin{proposition}\label{prop:lcfinitenessDrs}
   Let $M$ be a finitely generated $\Drs$-module. The $\Drs$-module $\rR^i\mGaq(M)$ is finitely generated for all $i$ and vanishes for sufficiently large $i$. 
\end{proposition}
\begin{proof}
By Theorem~\ref{thm:gensdbDrs}, it suffices to prove the claim for semi-induced modules and torsion modules. For semi-induced modules, this is Corollary~\ref{cor:lcvanishing}, and for torsion modules, this is Proposition~\ref{prop:torsionlc}.
\end{proof}

\subsubsection*{Local cohomology modules supported at other primes}
Recall that $u = p^d$ with $r \leq d \leq s < \infty$. 

It is now essentially a formal consequence of Property (Inj) to show that the local cohomology modules supported at primes of height $>1$ also preserve finite generation. The remainder of this section is devoted to prove this. We freely use the language and results of \cite[Section~4]{ss19gl2}.
\begin{lemma}\label{lem:restgamma}
    Let $\fa \subset \fb$ be $\GL$-prime ideals of $A$, and consider the two Serre subcategories
    \[  \cC \coloneqq (\Mod_A[\fa])[\fb^{\infty}] \to \Mod_A[\fa] \]
    and
    \[ \cD \coloneqq \Mod_A[\fb^{\infty}] \to \Mod_A[\fa^{\infty}] \] with right adjoints $\overline{\Gamma}$ and $\Gamma$ respectively. Assume further that Property (Inj) holds for relevant Serre subcategories, i.e., injectives in $\cC$ remain injective in $\Mod_A[\fa]$, injectives in $\cD$ remain injective in $\Mod_A[\fa^{\infty}]$, and injectives in $\Mod_A[\fa^{\infty}]$ remain injective in $\Mod_A$. Let $\iota \colon \Mod_A[\fa] \to \Mod_A[\fa^{\infty}]$ be the restriction functor and $\overline\iota \colon \cC \to \cD$ the induced restriction. We have an isomorphism $\overline\iota \circ \rR\overline\Gamma \cong \rR \Gamma \circ \iota$.
\end{lemma}
\begin{proof} 
    Let $I$ be an indecomposable injective object in $\Mod_A[\fa]$. Since the subcategory $\cC \subset \Mod_A[\fa]$ satisfies Property (Inj), the injective $I$ is either an indecomposable injective object in $\cC$ or isomorphic to $S(J)$ for some indecomposable injective $J$ in $\Mod_A[\fa]/\cC$. In the former case, since the subcategory $\cD \subset \Mod_A[\fa^{\infty}]$ satisfies Property (Inj), we have $\rR^i \Gamma(\iota(I)) = 0$. Assume instead that we are in the latter case. By derived adjunction $\rR \Hom_{A}(T, \iota(I)) \cong \rR \Hom_{A/\fa}(T \stackrel{\rL}{\otimes}_A A/\fa, I) $.  When $T$ is $\fb^{\infty}$-torsion, then $T \stackrel{\rL}{\otimes}_A  A/\fa$ is also $\fb^{\infty}$-torsion, and since $I \cong S(J)$, we see that $\rR \Hom_{A/\fa}(T \stackrel{\rL}{\otimes}_A A/\fa, I) = 0$ and thus the object $\iota(I)$ satisfies $\rR\Hom_{A}(T, \iota(I)) = 0$. So for any injective $I$, the restriction $\iota(I)$ is $\Gamma$-acyclic. The result follows.
\end{proof}
\begin{lemma}\label{lem:semilcvan}
   For a semi-induced $\Drs$-module $M$, we have $\rR \mGau(M) = 0$. 
\end{lemma}
\begin{proof}
   By Corollary~\ref{cor:semiinj}, the module $M$ has an injective resolution by torsion-free $\Drs$-injectives. Applying the functor $\mGau$, the whole complex vanishes so in particular the local cohomology modules vanish.
\end{proof}

 \begin{proposition}\label{prop:lcgauDrs}
 The functor $\rR\mGau$ maps the bounded derived category $\rD^b_{\fgen}(\Mod_{\Drs})$ to $\rD^b_{\fgen}(\Mod_{\Drs}[\Jrd^{\infty}])$.
 \end{proposition}
 \begin{proof}
 We proceed by induction on $d - r$. When $d - r = 0$, this follows by  Proposition~\ref{prop:lcfinitenessDrs} and \cite[Lemma~6.5]{gan25pol}. Now assume $d - r > 0$.  By Theorem~\ref{thm:gensdbDrs}, it suffices to show that $\rR \mGau(\Drs/J_{[r,e]} \otimes L_{\lambda})$ and $\rR \mGau(\Drs \otimes L_{\lambda})$ lie in $\rD^b_{\fgen}(\Mod_{\Drs}[\Jrd^{\infty}])$ for all integers $r \leq e \leq s$ and all partitions $\lambda$. Lemma~\ref{lem:semilcvan} gives this result for $\rR \mGau(\Drs \otimes L_{\lambda})$. The modules $\Drs/J_{[r,e]} \otimes L_{\lambda}$ are $D_{[r+1, s]}$-modules, so by the induction hypothesis and Lemma~\ref{lem:restgamma}, we get that $\rR^i\mGau(\Drs/J_{[r,e]} \otimes L_{\lambda})$ is a finitely generated $\Drs$-module and vanishes for sufficiently large $i$, as required.
 \end{proof}

The above results in particular imply the following semi-orthogonal decomposition. We merely provide helpful references since we will not use this result.

\begin{theorem}\label{thm:sodDrs} 
Let $\cD$ be the bounded derived category of finitely generated $\Drs$-modules. For $r \leq d \leq s$, let $\cT_d \subset \cD$ be the full triangulated subcategory generated by $D_{[d, s]} \otimes L_{\lambda}$; and let $\cT_{s+1} \subset \cD$ be the full triangulated subcategory generated by the finite-length modules. We have a semi-orthogonal decomposition
   \[
   \cD = \langle \cT_{s+1}, \cT_s, \ldots, \cT_r \rangle.
   \]
\end{theorem}
\begin{proof}
    How this follows from the above results is explained in Section~4.2 and Section~4.3 of \cite{ss19gl2}. 
\end{proof}

\section{Finitely presented \texorpdfstring{$D$}{D}-modules}\label{s:fpd}
In this section, we prove the remaining results from the introduction.
\subsection{Comparison results}\label{ss:comparison}
Let $d$ be a non-negative integer with $r \leq d$  and set $u = p^d$. We denote by $\mGau$ the functor that sends a $\Dr$-module to its $J_{[r,d]}^{\infty}$-torsion submodule.

\begin{lemma}\label{lem:restinj}
     Assume $I$ is an injective object in $\Mod_{\Dr}$ and $r \leq s < \infty$. Under restriction of scalars, the module $I$ is also injective in $\Mod_{\Drs}$.
 \end{lemma}
 \begin{proof}
     Restriction of scalars is right adjoint to the exact functor $\Drinfty \otimes_{\Drs}$, so injectives remain injective under restriction.
 \end{proof}

    The next result shows that the local cohomology modules with support in $\Jrd$ of a $\Drinfty$-module $M$
    may essentially be computed in $\Drs$ for $s \geq d$.
 \begin{lemma}\label{lem:drinflc}
    Assume $M$ is a $\Drinfty$-module. For any $s \geq d$, the restriction along the map $\Drs \to \Drinfty$ of the local cohomology modules $\rR^i \mGau(M)$ computed in $\Mod_{\Drinfty}$ agrees with the local cohomology $\rR^i \mGau(M)$ computed in $\Mod_{\Drs}$ after first restricting $M$ along the map $\Drs \to \Drinfty$.
 \end{lemma}
 \begin{proof} 
     Let $F \colon \Mod_{\Drinfty} \to \Mod_{\Drs}$ be the restriction functor. It is easy to verify that $F \circ \mGau \cong \mGau \circ F$. The result now follows by Lemma~\ref{lem:restinj}.
 \end{proof}

 \begin{proposition}\label{prop:torsionlcDr}
     Let $d \geq r$ be an integer and set $u = p^d$. Assume $M$ is a $\Drinfty$-module locally annihilated by $\Jrd$. The local cohomology modules $\rR^i\mGau(M) = 0$ for all $i > 0$ and $\mGau(M) = M$.
 \end{proposition}
 \begin{proof}
 We obviously have $\mGau(M) = M$.
 Restricting a module locally annihilated by $\Jrd$ to $D_{[r,t]}$ with $t \gg d$ remains locally annihilated by $\Jrd \subset D_{[r,t]}$. 
 Therefore, the higher local cohomology modules of $M$ computed after restricting to $\Mod_{D_{[r,t]}}$ for $t \gg d$ vanish by Proposition~\ref{prop:torsionlc}.
 We therefore obtain vanishing of the positive local cohomology modules computed in $\Drinfty$ by Lemma~\ref{lem:drinflc}.
 \end{proof}

 \begin{proposition}\label{prop:Drlcvanishing}
  Let $d \geq r$ be an integer and set $u = p^{d}$. Assume $M$ is a finitely presented $\Drinfty$-module filtered by modules of the form $D^{(d)} \otimes L_{\lambda}$ with $\lambda$ allowed to vary. Then $\rR^i \mGau(M) = 0$ for all $i \in \bN$.
 \end{proposition}
 \begin{proof}
     By a d\'evissage argument, we may assume $M \cong D^{(d)} \otimes L_{\lambda}$.
     It suffices to show that $\rR^i \mGau(M) = 0$ when computed after restricting to $\Drs$ with $s \gg d$ by Lemma~\ref{lem:drinflc}, and by Lemma~\ref{lem:restgamma}, it suffices to show $\rR^i \mGau(M) = 0$ when computed in $\Mod_{D_{[d,s]}}$. 
     The $D_{[d,s]}$-module $M^{<n}$ is flat since it is free upon forgetting the $\GL$-action, hence it is semi-induced by Proposition~\ref{prop:flatequalssemi}. So $\rR^i\mGau(M^{<n}) = 0$ for all $n$ by Lemma~\ref{lem:semilcvan}, hence $\rR^i\mGau(M) = \rR^i\mGau(\colim M^{<n}) = \colim \rR^i \mGau(M^{<n}) = \colim 0 = 0$ where the second equality uses Lemma~\ref{lem:lccolim}. 
 \end{proof}

\begin{lemma}\label{lem:inducedbasechange}
    Assume $M$ is a $\Drs$-module. The natural map 
    $$\Drinfty \otimes_{\Drs} \mShq(M) \to \mShq(\Drinfty \otimes_{\Drs} M)$$ 
    is an isomorphism. 
\end{lemma}
\begin{proof}
When $M$ is induced, the natural map is easily verified to be an isomorphism. For arbitrary $M$, take a presentation
\[
0 \to K \to F \to M \to 0
\]
with $F$ induced. In the diagram
\[
\begin{tikzcd}[column sep=small]
    0 \arrow[r] & \Drinfty \otimes_{\Drs} \mShq(K) \arrow[d] \arrow[r] & \Drinfty \otimes_{\Drs} \mShq(F) \arrow[d] \arrow[r] &  \Drinfty \otimes_{\Drs} \mShq(M) \arrow[d] \arrow[r] & 0 \\
    0 \arrow[r] & \mShq(\Drinfty \otimes_{\Drs} K)  \arrow[r] & \mShq(\Drinfty \otimes_{\Drs} F) \arrow[r] &  \mShq(\Drinfty \otimes_{\Drs} M) \arrow[r] & 0 
\end{tikzcd}
\]
the middle vertical map is an isomorphism, so by the snake lemma, the rightmost vertical map is surjective. That is, the natural map is surjective. So the leftmost vertical map is also an isomorphism. Thus, the rightmost vertical map is also injective. That is, the natural map is also injective, as required.
\end{proof}

\subsection{Proof of Theorem~\ref{thm:gensdbmodintro}}\label{ss:proofs}
In this section, we obtain generators of the bounded derived category of finitely presented $\Dr$-modules after proving a shift theorem.

\begin{theorem}[Shift theorem for $\Dr$]\label{thm:shiftDr}
   Assume $M$ is a finitely presented $\Dr$-module. For sufficiently large $t$, the $\Dr$-module $\mShq^t(M)$ is semi-induced.
\end{theorem}
\begin{proof}
    By Lemma~\ref{lem:finitebase}, there exists $s < \infty$ and a finitely generated $\Drs$-module $N$ such that $\Drinfty \otimes_{\Drs} N \cong M$. By the shift theorem for $\Drs$, the $\Drs$-module $\mShq^t(N)$ is semi-induced for $t \gg 0$. By Lemma~\ref{lem:inducedbasechange}, $\mShq^t(M)$ is isomorphic to $\Drinfty \otimes_{\Drs} \mShq^t(N)$. The shift theorem  follows since the base change of a semi-induced $\Drs$-module to $\Drinfty$ is a semi-induced $\Drinfty$.
\end{proof}

\begin{proposition}[Resolution Theorem]\label{prop:resolutionDr}
Let $M$ be a finitely presented $\Dr$-module. We have a chain complex of $\Dr$-modules
\begin{displaymath}
0 \to M \to P^0 \to P^1 \to \ldots \to P^m \to 0
\end{displaymath}
satisfying the following properties:
\begin{itemize}
\item each $P^i$ is a finitely presented semi-induced module with $t_0(P^i) \le t_0(M) - qi$, and
\item the cohomology of this complex is a torsion $\Dr$-module.
\end{itemize}
{Furthermore, given a map of $\Dr$-modules $f \colon M \to N$, we can choose complexes $M \to P^{\bullet}$ and $N \to Q^{\bullet}$ satisfying the above properties, and a map of complexes $\tilde{f}\colon P^{\bullet} \to Q^{\bullet}$ extending $f$.}
\end{proposition}
\begin{proof}
By Lemma~\ref{lem:finitebase}, there exists $t < \infty$ and a finitely generated $D_{[r, t]}$-module $N$ such that $M$ is obtained by flat base change along the extension $D_{[r, t]} \to \Drinfty$. The existence of the analogous complex for $M$ follows by taking the complex for $N$, which exists by Proposition~\ref{prop:resolutionDrs}, and applying the functor $\Drinfty \otimes_{D_{[r,t]}} (-)$. That the resulting complex satisfies the requisite properties follows by combining flatness of $D_{[r, t]} \to \Drinfty$, Lemma~\ref{lem:torsext}, and Lemma~\ref{lem:extension}.
\end{proof}

\begin{proposition}\label{prop:trianglesDr}
    Let $M$ be a finitely presented $\Dr$-module. We have a triangle $$T \to M \to F \to$$ in $\rD^b_{\fpre}(\Mod_{\Dr})$ where $T$ is quasi-isomorphic to a bounded complex of finitely presented torsion $\Dr$-modules, and $F$ is quasi-isomorphic to a bounded complex of finitely generated semi-induced $\Dr$-modules. 
\end{proposition}
\begin{proof} 
This is a routine application of Lemma~\ref{lem:finitebase} as in the proof of Proposition~\ref{prop:resolutionDr} using Proposition~\ref{prop:triangles}.
\end{proof}

We can now prove the result on generators of the derived category. We state a more general version of Theorem~\ref{thm:gensdbmodintro}, which is the $r=0$ case of this next result.

\begin{theorem} \label{thm:gensdbDr}
   The bounded derived category $\rD^b_{\fpre}(\Mod_{\Dr})$ is generated by the modules $D^{(d)} \otimes L_{\lambda}$ with $d \geq r$ and $\lambda$ varying over all partitions. 
\end{theorem}
\begin{proof}
    It suffices to show that the class of every finitely presented $\Drinfty$-module $M$ lies in the triangulated subcategory generated by the modules $D_{[t, \infty]} \otimes L_{\lambda}$ with $t \geq r$ and $\lambda$ varying over all partitions. By Lemma~\ref{lem:finitebase}, there exists $s < \infty$ and a finitely generated $\Drs$-module $N$ such that $\Drinfty \otimes_{\Drs} N \cong M$. In $\rD^b_{\fgen}(\Mod_{\Drs})$, the module $N$ lies in the triangulated subcategory generated by the modules $L_{\lambda}$ and $D_{[t, s]} \otimes L_{\lambda}$ with $t$ allowed to vary. By flatness of $D_{[r, s]} \to D_{[r, \infty]}$ and Lemma~\ref{lem:extension}, we see that $\Drinfty \otimes N \cong M$ lies in the triangulated subcategory generated by $D_{[t, \infty]} \otimes L_{\lambda}$ with $r \leq t \leq s+1$, as required.
\end{proof}

\subsection{Proof of Theorem~\ref{thm:introsod}}
In this section, we focus our attention on the derived category of $D$-modules.

Let $\cT$ be the bounded derived category of finitely presented $D$-modules; for a $\GL$-prime ideal $\fa \subset D$, let $\cT[\fa^{\infty}]$ be its triangulated subcategory containing complexes with bounded homology in $\Mod_D[\fa^{\infty}]$. 
Let $\cT^+$ be the (unbounded) derived category of $D$-modules; we define $\cT^+[\fa^{\infty}]$ analogously.

For a prime power $u = p^d$, we let $\mGau \colon \Mod_D \to \Mod_D[J_{[0,d]}^{\infty}]$ be the functor that maps a $D$-module to its submodule of elements locally annihilated by $J_{[0,d]}$.
This functor is right adjoint to the inclusion functor.
We let $\rR \mGau \colon \cT^+ \to \cT^+[J_{[0,d]}^{\infty}]$ be the derived functor of $\mGau$. 

\begin{lemma}\label{lem:equiv}
Let $d \in \bN$. The natural functor
\[\rD^b_{\fpre}(\Mod_D[J_{[0,d]}^{\infty}]) \to \cT[J_{[0,d]}^{\infty}] \] 
is an equivalence.
\end{lemma}
\begin{proof}
    Let $0 \to A \to B \to C \to 0$ be an exact sequence of finitely presented $D$-modules with $A$ being $J_{[0,d]}^{\infty}$-torsion. There exists $s \gg d$ and an exact sequence $0 \to A' \to B' \to C' \to 0$ of finitely generated $D_{[0,s]}$-modules such that the original exact sequence is obtained by base changing this one along the inclusion $D_{[0,s]} \to D$. It is easy to see that $A'$ is locally annihilated by $J_{[0,d]} \subset D_{[0,s]}$. Let $I$ be an injective $D_{[0,s]}$-module containing $A'$; by Proposition~\ref{prop:propinj}, we may choose $I$ to be $J_{[0,d]}^{\infty}$-torsion. The map $A' \to B'$ induces a map $B' \to I$; let $I' \subset I$ be a finitely generated subobject containing the image of $B' \to I$. We thus get a map of short exact sequences
   \[\begin{tikzcd}[column sep=small]
	0 & {A'} & {B'} & {C'} & 0 \\
	0 & {A'} & {I'} & {I'/A'} & 0
	\arrow[from=1-1, to=1-2]
	\arrow[from=1-2, to=1-3]
	\arrow[equals, from=1-2, to=2-2]
	\arrow[from=1-3, to=1-4]
	\arrow[from=1-3, to=2-3]
	\arrow[from=1-4, to=1-5]
	\arrow[from=1-4, to=2-4]
	\arrow[from=2-1, to=2-2]
	\arrow[from=2-2, to=2-3]
	\arrow[from=2-3, to=2-4]
	\arrow[from=2-4, to=2-5]
\end{tikzcd},\]
which tensored with $D \otimes_{D_{[0,s]}}$ results in 
  \[\begin{tikzcd}[column sep=small]
	0 & {A} & {B} & {C} & 0 \\
	0 & {A} & D \otimes_{D_{[0,s]}} {I'} & D \otimes_{D_{[0,s]}} {I'/A'} & 0
	\arrow[from=1-1, to=1-2]
	\arrow[from=1-2, to=1-3]
	\arrow[equals, from=1-2, to=2-2]
	\arrow[from=1-3, to=1-4]
	\arrow[from=1-3, to=2-3]
	\arrow[from=1-4, to=1-5]
	\arrow[from=1-4, to=2-4]
	\arrow[from=2-1, to=2-2]
	\arrow[from=2-2, to=2-3]
	\arrow[from=2-3, to=2-4]
	\arrow[from=2-4, to=2-5]
\end{tikzcd}.\]
The $D$-module $D\otimes_{D_{[0,s]}} I'$ is finitely presented and $J_{[0,d]}^\infty$-torsion. The lemma now follows from \cite[1.15 Lemma c(1)]{kel99ex}.
\end{proof}

\begin{remark}\label{rmk:subtleties}
  We do not know whether the functor $\cT^+[J_{[0,d]}^{\infty}] \to \cT^+$ is fully faithful. This holds if the Serre subcategory $\Mod_D[J_{[0,d]}^{\infty}] \subset \Mod_D$ satisfies Property (inj). We can only prove that an injective object $I$ in $\Mod_D[J_{[0,d]}^{\infty}]$ is FP-injective in $\Mod_D$; i.e., for \textit{finitely presented} objects $M$, we have $\ext^i(M, I) = 0$ for all $i > 0$.
\end{remark}

\begin{proposition}\label{prop:maintriangles}
\leavevmode
    \begin{enumerate}
    \item For each complex $C$ in $\cT$, there exists a triangle 
    $T \to C \to F \to$ where $T$ is an object in $\cT[J_{[0,0]}^{\infty}]$ and $F$ is a complex of objects filtered by induced $D$-modules. 
    \item For all non-negative integers $r$, given a complex $C$ in $\cT[J_{[0,r]}^{\infty}]$, there exists a triangle $T \to C \to F \to$ 
    where $T$ is a complex in $\cT[J_{[0,r+1]}^{\infty}]$ and $F$ is a complex of modules filtered by $D^{(r+1)} \otimes L_{\lambda}$ with $\lambda$ varying. 
    \end{enumerate}
\end{proposition}
\begin{proof}
(1)  The full subcategory on the objects for which the conclusion of the first part holds is a triangulated subcategory of $\cT$. By Proposition~\ref{prop:trianglesDr}, this contains the class of every finitely presented $D$-module, which generates all of $\cT$, as required.

(2) Similar to the previous part, the set of objects of $\cT[J_{[0,r]}^{\infty}]$ for which the conclusion of the second part holds is also a triangulated subcategory of $\cT[J_{[0,r]}^{\infty}]$. By the previous lemma, this triangulated category is generated by finitely presented modules locally annihilated by $J_{[0,r]}$ which are filtered by finitely presented modules over $D/J_{[0,r]} \cong D^{(r+1)}$. By Proposition~\ref{prop:trianglesDr}, given a finitely presented $D^{(r+1)}$-module $M$, there exists a triangle $T \to M \to F \to$ with $T$ being a complex of torsion $D^{(r+1)}$-modules and $F$ being a complex of modules filtered by induced $D^{(r+1)}$-modules. By restriction, we see that every $D^{(r+1)}$-module satisfies the conclusion of the proposition, whence the result follows.
\end{proof}

\begin{lemma}\label{lem:orth}
    Let $\lambda$ be a partition and $r \geq 0$. For all complexes $T$ in $\cT[J_{[0,r]}^{\infty}]$, we have $\Hom_\cT(T, D^{(r)} \otimes L_{\lambda}) = 0$. 
\end{lemma}
\begin{proof}
    It suffices to prove that $\ext^i_D(T, D^{(r)}\otimes L_{\lambda}) = 0$ for all $J_{[0,r]}^{\infty}$-torsion $D$-modules and $i \geq 0$. Let $D^{(r)}\otimes L_{\lambda} \to I^{\bullet}$ be an injective resolution of $D^{(r)} \otimes L_{\lambda}$. We have $\Hom_D(T, I^{\bullet}) \cong \Hom_{D}(T, \Gamma_{p^r}(I^{\bullet}))$ by adjunction. 
    The complex $\Gamma_{p^r}(I^{\bullet})$ is a complex of injectives in the subcategory $\Mod_D[J_{[0,r]}^{\infty}]$ as $\Gamma_{p^r}$ is right adjoint to the (exact) inclusion functor; the homology of this complex is $\rR^i\Gamma_{p^r}(D^{(r)} \otimes L_{\lambda})$ all of which vanish by Proposition~\ref{prop:Drlcvanishing}. Thus it is a nullhomotopic complex in $\Mod_D[J_{[0,r]}^{\infty}]$  and thus remains acyclic after applying $\Hom(T,-)$, which implies $\ext^i(T, D^{(r)} \otimes L_{\lambda})=0$ for all non-negative $i$, as required.
\end{proof}

The existence of the semi-orthogonal decomposition is now quite formal.

\begin{proof}[Proof of Theorem~\ref{thm:introsod}]
   By Part (1) of Proposition~\ref{prop:maintriangles}, we see that every object $M$ in $\cT$ fits into a triangle $T \to M \to F \to$ with $T$ in $\cT[J_{[0,0]}^{\infty}]$ and $F$ in the right orthogonal of $\cT[J_{[0,0]}^{\infty}]$ by the preceding lemma. By \cite[Proposition~4.9.1]{kra10loc}, this implies that the inclusion functor $\cT[J_{[0,0]}^{\infty}] \to \cT$ has a right adjoint (i.e., the local cohomology functor $\rR \Gamma_1$ maps the bounded derived category $\cT$ to $\cT[J_{[0,0]}^{\infty}]$) giving a semi-orthogonal decomposition 
    \[ \cT = \langle \cT[J_{[0,0]}^{\infty}], \cT_0 \rangle.\]
    The category $\cT_0$ consists of objects $Y$ which satisfy $\Hom_\cT(X, Y) = 0$ for all objects $X \in \cT[J_{[0,0]}^{\infty}]$. Let $Z$ be such an object; by part (1) of  Proposition~\ref{prop:maintriangles}, we have a triangle $T \to Z \to F \to$, but since $\Hom_{\cT}(T, Z) = 0$ for all $T \in \cT[J_{[0,0]}^\infty]$, we see that $T=0$ and $Z \cong F$, as required.
    
    The identical argument as in the previous paragraph (but instead using part (2) of Proposition~\ref{prop:maintriangles}) shows that we have a semi-orthogonal decomposition
    \[ \cT[J_{[0,r]}^{\infty}] = \langle \cT[J_{[0,r+1]}^{\infty}], \cT_{r+1} \rangle, \]
    where $\cT_{r+1}$ is generated by the classes of the modules $D^{(r+1)} \otimes L_{\lambda}$ with $\lambda$ varying over all partitions. 
    
    The semi-orthogonal decompositions in the statement of theorem is obtained by putting together all these semi-orthogonal decompositions.
\end{proof}

\subsection{Complement: Frobenius twists of exterior algebras}
We sketch how to incorporate ideas from this paper to our earlier work \cite{gan22ext} to study Frobenius twists of the exterior algebra. As explained in \cite[Postlude]{gan24lnnr}, we expect these algebras to play a role in Koszul duality like results for $S$-modules.

Let $R$ be the exterior algebra $\lw\{\bV\}$ as analyzed in \cite{gan22ext}. Fix an integer $r \geq 1$ and let $q = p^r$. 
   \begin{theorem}
   \leavevmode
   \begin{enumerate}
       \item The $\GL$-spectrum of $R^{(r)}$ contains two points: the zero ideal and the homogeneous maximal ideal.
       \item The category of $\GL$-equivariant $R^{(r)}$-modules is locally noetherian.
       \item Assume $M$ is a finitely generated $R^{(r)}$-module. For sufficiently large $t$, the $R^{(r)}$-module $\mShq^t(M)$ is semi-induced.
       \item We have a semi-orthogonal decomposition 
       $$\rD^b_{\fgen}(\Mod_{R^{(r)}}) = \langle  \cT, \cG \rangle$$
       where $\cT$ is the triangulated category generated by objects of finite length and $\cG$ is the triangulated category generated by objects of the form $R^{(r)} \otimes L_{\lambda}$ with $\lambda$ allowed to vary over all partitions.
   \end{enumerate}
   \end{theorem}

(1) above is clear by irreducibility of Frobenius twists of exterior powers of the standard representation $\bV$ of $\GL$; (2) essentially follows from Cohen's theorem after restricting to the permutation subgroup of $\GL$.  The key technical input required for (3) is a variant of Lemma~\ref{lem:newlemma} which we leave as an exercise to the reader as it is entirely analogous to the proofs in Section~\ref{ss:technical}.
\begin{lemma}\label{lem:Rnew}
    Assume $W$ is a polynomial representation of $\GL$ and $t > 0$. Given a nonzero subrepresentation $U \subset (\lw^t\{\bV\})^{(r)} \otimes W$, the $q$-th Hasse--Schur derivative $\mShq(U) \ne 0$. 
\end{lemma}
Given this lemma, all results of \cite{gan22ext} hold with statements and proofs modified appropriately. In particular, so do (3) and (4). Lemma~\ref{lem:Rnew} is required in Proposition~4.10 of ibid. (for $r = 0$, we originally just cited the Steinberg tensor product theorem).

We expect an analogous argument to give a parallel theory of $\Sym\{\bV\}^{(r)}$-modules, completing the story for Frobenius twists of these $\GL$-algebras; we leave the details to future work. 
\bibliographystyle{alpha}
\bibliography{bibliography}
\end{document}